\documentclass[11pt]{amsart}
\usepackage{amssymb}
\usepackage{graphicx}
\usepackage{xcolor} 
\usepackage{tensor}
\usepackage{fullpage} 
\usepackage{amsmath}
\usepackage{amsthm}
\usepackage{verbatim}
\usepackage{hyperref}
\usepackage{array} 
\usepackage{enumitem}
\setlist[enumerate]{leftmargin=1.2em}
\setlist[itemize]{leftmargin=1.2em}
\usepackage{seqsplit,mathtools}
\mathtoolsset{showonlyrefs}

\definecolor{green}{rgb}{0,0.5,0} 

\newcommand{\Red}[1]{\begingroup\color{red} #1\endgroup} 
\newcommand{\Blue}[1]{\begingroup\color{blue} #1\endgroup} 
\newcommand{\Green}[1]{\begingroup\color{green} #1\endgroup}

\newtheorem{theorem}{Theorem}[section]
\newtheorem{thm}{Theorem}[section]

\newtheorem{lem}[theorem]{Lemma}

\newtheorem{prop}[theorem]{Proposition}

\theoremstyle{definition}

\theoremstyle{remark}
\newtheorem{remark}[theorem]{Remark}

\numberwithin{equation}{section}

\numberwithin{equation}{section}
\newcommand{\nrm}[1]{\Vert#1\Vert}

\newcommand{\br}[1]{\overline{#1}}

\newcommand{\nnrm}[1]{{\vert\kern-0.25ex\vert\kern-0.25ex\vert #1 
		\vert\kern-0.25ex\vert\kern-0.25ex\vert}}

\newcommand{\lap}{\Delta}

\newcommand{\rd}{\partial}
\newcommand{\nb}{\nabla}

\newcommand{\ift}{\infty}

\newcommand{\alp}{\alpha}
\newcommand{\bt}{\beta}
\newcommand{\gmm}{\gamma}

\newcommand{\veps}{\varepsilon}

\newcommand{\tht}{\theta}

\newcommand{\omg}{\omega}

\newcommand{\bfe}{{\bf e}}

\newcommand{\bfu}{{\bf u}}
\newcommand{\bfv}{{\bf v}}

\newcommand{\bfx}{{\bf x}}

\newcommand{\bbR}{\mathbb R}

\newcommand{\calA}{\mathcal A}

\newcommand{\ackn}[1]{
	\addtocontents{toc}{\protect\setcounter{tocdepth}{1}}
	\subsection*{Acknowledgements} {#1}
	\addtocontents{toc}{\protect\setcounter{tocdepth}{1}} }

\begin{document}
	
	\bibliographystyle{plain}
	\title{On the upper bound for the vorticity growth of bi-rotational Euler flows without swirl}

	\renewcommand{\thefootnote}{\fnsymbol{footnote}}
	\renewcommand{\thefootnote}{\arabic{footnote}}
	
	\author{Khakim Egamberganov}
	\address{Department of Mathematics, National University of Singapore, Block S17, 10 Lower Kent Ridge Road, Singapore, 119076, Singapore}
	\email{k.egamberganov@u.nus.edu}
	
	
	\author{Deokwoo Lim}
	\address{School of Mathematics, Korea Institute for Advanced Study, 85 Hoegi-ro, Dongdaemun-gu, Seoul, 02455, Republic of Korea.}
	\email{dwlim95@kias.re.kr}
	
	
	\date\today
	\maketitle
	
\begin{abstract}
		For $d\geq 4$, we consider incompressible Euler flows in $\bbR^{d}$ with bi-rotational symmetry and without swirl. Our first result gives the local wellposedness of the Yudovich-type solution. The second result provides global wellposedness up to $d\leq 6$. In particular, it shows that the rate of growth of the vorticity maximum coincides with the rate from axisymmetric flows without swirl, which was obtained in the paper by the second author and Jeong (Arch. Ration. Mech. Anal. 249(3):32, 2025) and Shao--Wei--Zhang (Acta Math. Sin. (Engl. Ser.), 42(3):663-679, 2026).
	\end{abstract}

	\medskip

	\section{Introduction}\label{sec:intro}
	
	The incompressible Euler equations in $\bbR^d$ are given by
	\begin{equation}\label{eq:Eulereq}
		\left\{
		\begin{aligned}
			\rd_{t} \bfv +  \bfv \cdot\nb \bfv +\nb p&=0,\\
			\nb\cdot \bfv &=0, 
		\end{aligned}
		\right.
	\end{equation}
	where $ \bfv = (v^{1},\cdots,v^{d}) : [0,T)\times\bbR^{d}\to\bbR^{d} $ and $ p : [0,T)\times\bbR^{d}\to\bbR $. The coordinates that we are going to consider is the \textit{bi-polar coordinates,} which is given as
	\begin{equation*}
		\begin{aligned}
			r^{2} = \sum_{i=1}^{n+1}x_{i}^{2}, \qquad s^{2} = \sum_{i=n+2}^{n+m+2}x_{i}^{2}. 
		\end{aligned}
	\end{equation*}
	Under this bi-polar coordinates, we say that a solution $ \bfv $ of \eqref{eq:Eulereq} has \textit{bi-rotational symmetry without swirl} if it has the form
	\begin{equation}\label{eq:birothighd}
		\bfv(t,\bfx) = u^{r}(t,r,s)\bfe_{r} + u^{s}(t,r,s)\bfe_{s}, \qquad \bfx = (x_{1},...,x_{d}), 
	\end{equation} where $u^{r}, u^{s}$ are scalar-valued. 
	
	Recall that in $\bbR^{d}$, the vorticity $ \boldsymbol{\omg} $ is a two-tensor of the form
	\begin{equation}\label{eq:vorttensor}
		\boldsymbol{\omg} = (\omg^{i,j})_{1\leq i,j\leq d},\quad \omg^{i,j}=\rd_{j}v^{i}-\rd_{i}v^{j}.
	\end{equation} 
	Similarly, the stream function $ (\Psi^{i,j})_{1\leq i,j\leq d} $ is defined as
	\begin{equation}\label{eq:streamftntensor}
		\Psi^{i,j}=\lap^{-1}\omg^{i,j}
	\end{equation}
	Under the bi-rotational symmetry and no-swirl assumption \eqref{eq:birothighd}, the vorticity and the stream function have the expression
	\begin{equation}\label{eq:vorttensca}
		\omg^{i,j}(\bfx)=-\frac{x_{i}x_{j}}{rs}w(r,s),\qquad w=\rd_{r}u^{s}-\rd_{s}u^{r}, \qquad \Psi^{i,j}(\bfx)=-\frac{x_{i}x_{j}}{rs}\psi(r,s),
	\end{equation}
	for some scalar functions $w$ and $\psi$, and for any $1\leq i\leq n+1$, $n+2\leq j\leq n+m+2$, and $ \bfx\in \bbR^{d} $. From now on, we refer to $w$ and $\psi$ as the scalar vorticity and the scalar stream function, respectively. Then, the scalar vorticity $w$ satisfies the following evolution equation:
	\begin{equation}\label{eq:vorteq}
		\rd_{t}w+(u^{r}\rd_{r}+u^{s}\rd_{s})w=\bigg(n\frac{u^{r}}{r}+m\frac{u^{s}}{s}\bigg)w,
	\end{equation}
	Alternatively, this can be written as the transport equation of the quantity $w/(r^{n}s^{m})$ in $\bbR^{d}$:
	\begin{equation}\label{eq:vorticityeq}
		\rd_{t}\left(\frac{w}{r^{n}s^{m}}\right)+(u^{r}\rd_{r}+u^{s}\rd_{s})\left(\frac{w}{r^{n}s^{m}}\right)=0,
	\end{equation}
	Also, from the divergence-free condition
	\begin{equation}\label{eq:divergencefree}
		\rd_{r}u^{r}+\frac{n}{r}u^{r}+\rd_{s}u^{s}+\frac{m}{s}u^{s}=0
	\end{equation}
	and the relation between $w$ and the scalar stream function $\psi$
	\begin{equation}\label{eq:psi-w-relation}
		\left\{
		\begin{aligned}
			\bigg(\rd_{r}^{2}+\frac{n}{r}\rd_{r}-\frac{n}{r^{2}}+\rd_{s}^{2}+\frac{m}{s}\rd_{s}-\frac{m}{s^{2}}\bigg) \psi = w \quad &\mbox{in} \quad \Pi , \\
			\psi = 0 \quad &\mbox{on} \quad \partial\Pi = \left\{ (r,s) \, : \, r = 0 \mbox{ or } s = 0 \right\}, 
		\end{aligned}
		\right.
	\end{equation}
	we can get the bi-rotational velocity $\bfu:=(u^{r},u^{s})$ from $\psi$ as
	\begin{equation}\label{eq:psicurl}
		u^{r}=-\frac{\rd_{s}(s^{m}\psi)}{s^{m}}=-\bigg(\frac{m}{s}\psi+\rd_{s}\psi\bigg),\quad u^{s}=\frac{\rd_{r}(r^{n}\psi)}{r^{n}}=\frac{n}{r}\psi+\rd_{r}\psi.
	\end{equation}
	In particular, $\psi$ can be recovered from the relation \eqref{eq:psi-w-relation} as
	\begin{equation}\label{eq:psiinvlap}
		\psi
		=rs\lap_{\bbR^{n+3}\times\bbR^{m+3}}^{-1}\bigg[\frac{w}{rs}\bigg],
	\end{equation}
	where $\lap_{\bbR^{n+3}\times\bbR^{m+3}}$ is the Laplacian in $\bbR^{d+4}=\bbR^{n+3}\times\bbR^{m+3}$.
	
	\subsection{Main results}
	
	The first result is about the local wellposedness of the Yudovich-type solutions.
	\begin{thm}\label{thm:LWP}
		Let $w_{0}\in (L^{d,1}\cap L^{\ift})(\bbR^{d})$ satisfy
		\begin{equation}
			(1+r^{n-1}s^{m}+r^{n}s^{m-1})\frac{w_{0}}{r^{n}s^{m}}\in L^{d,1}(\bbR^{d}).
		\end{equation}
		Then there exists $T>0$ and a corresponding unique solution $w\in L^{\ift}(0,T;(L^{d,1}\cap L^{\ift})(\bbR^{d}))$ of \eqref{eq:Eulereq} that satisfies
		\begin{equation}
			(1+r^{n-1}s^{m}+r^{n}s^{m-1})\frac{w}{r^{n}s^{m}}\in L^{\ift}(0,T;L^{d,1}(\bbR^{d})).
		\end{equation}
	\end{thm}
	
	The second result says that for $d\leq 6$, under some additional conditions on the above initial data $w_{0}$, we have global wellposedness. Moreover, it shows that the $L^{\ift}$-norm of $w$ has the same the rate of growth with that from the axisymmetric flow without swirl (\cite{LJ_optimal, SWZ25}).
	\begin{thm}\label{thm:GWP}
		Let $w_{0}\in (L^{d,1}\cap L^{\ift})(\bbR^{d})$ also 
		satisfy $\bfv_{0}\in L^{2}(\bbR^{d})$ 
		and 
		\begin{equation}\label{eq:decaycondi}
			(1+r^{n}+s^{m})\frac{w_{0}}{r^{n}s^{m}}\in L^{\ift}(\bbR^{d}),\qquad (1+r^{n+m}+s^{n+m})\frac{w_{0}}{r^{n}s^{m}}\in L^{1}(\bbR^{d}).
		\end{equation} 
		Then there exists a unique global-in-time solution $w\in L_{\text{loc}}^{\ift}(0,\ift;(L^{d,1}\cap L^{\ift})(\bbR^{d}))$ of \eqref{eq:Eulereq} with the initial data $w_{0}$, and for any $t\geq0$, it satisfies
		\begin{equation} \label{eq:wttime}
			\nrm{w(t)}_{L^{\ift}(\bbR^{d})}\leq \begin{cases}
				C_{1}
				(1+t)^{4(d-2)/(6-d)},&\quad\text{if}\quad d=4,5,\\
				C_{2}
				e^{C_{3}
					t},&\quad\text{if}\quad d=6,
			\end{cases}
		\end{equation}
		for some constants $C_{1}=C_{1}(n,m,w_{0})>0$, $ C_{2}=C_{2}(n,m,w_{0})>0 $, and $C_{3}=C_{3}(n,m,w_{0})>0$.
	\end{thm}
	
	\begin{remark}
		The assumptions on $w_0$ in Theorem \ref{thm:GWP} can equivalently be stated as 
		\[ 
		w_{0}\in (L^{d,1}\cap L^{\ift})(\bbR^{d}), \qquad \frac{w_0}{r^ns^m}\in (L^1\cap L^{\infty}) (\mathbb{R}^d),
		\] 
		together with 
		\[
		\frac{w_0}{r^n}, \, \frac{w_0}{s^m} \in L^{\infty}(\mathbb{R}^d), \qquad \frac{r^m w_0}{s^m}, \, \frac{s^n w_0}{r^n} \in L^1 (\mathbb{R}^d).
		\]
		By interpolation, these assumptions imply those of Theorem \ref{thm:LWP}. 
	\end{remark}

	\begin{remark}
		For the case $d=4$, the first condition from \eqref{eq:decaycondi} is guaranteed if we assume $w_{0}\in C^{2}(\bbR^{4})$ (\cite[Prop. A.1]{CJL_gwpbi}). However, for cases $d=5,6$, smoothness of $w_{0}$ does not necessarily imply the $L^{\ift}$ condition in \eqref{eq:decaycondi}.
	\end{remark}

	\subsection{Literatures review}
	
	\subsubsection{Axisymmetric solutions in 3D}
	
	Bi-rotational solutions were motivated by axisymmetric solutions of 3D Euler equations. In $\bbR^{3}$, under the cylindrical coordinates $(r,\tht,z)$, \textit{axisymmetric solutions without swirl} are defined as
	\begin{equation}
		\bfu=u^{r}(r,z)\bfe_{r}+u^{z}(r,z)\bfe_{z},
	\end{equation}
	where the vorticity becomes
	\begin{equation}
		\boldsymbol{\omg}=\omg^{\tht}(r,z)\bfe_{\tht},\quad \omg^{\tht}(r,z)=\rd_{r}u^{z}-\rd_{z}u^{r},
	\end{equation}
	and the Euler equations reduce to
	\begin{equation}
		\rd_{t}\frac{\omg^{\tht}}{r}+(u^{r}\rd_{r}+u^{z}\rd_{z})\frac{\omg^{\tht}}{r}=0.
	\end{equation}
	The global wellposedness of the solution for this equation is well-known from various works (\cite{UY1968, Raymond, SY, Danaxi, AHK}). For such globally wellposed solutions, a natural question one can ask is whether the vorticity can grow as time $t$ goes to infinity. As the term $\omg^{\tht}/r$ is conserved along the flow, the growth of vorticity occurs when the radial distance of fluid particles from the rotation axis increases. The simplest model where such growth can be observed is \textit{the anti-parallel vortex rings}, which satisfy the following conditions:
	\begin{equation}\label{eq:antiparallel}
		\omg^{\tht}(r,-z)=-\omg^{\tht}(r,z),\quad \omg^{\tht}(r,z)\leq 0,\quad r,z\geq0.
	\end{equation}
	This model was considered by Childress, Gilbert, and Valiant (\cite{Child07, Child08, ChilGil1}).
	
	There are several works regarding the upper bound and the lower bound of the growth of the vorticity maximum $\nrm{\omg^{\tht}}_{L^{\ift}}$. For the upper bound, the exponential growth results can be found in Majda--Bertozzi \cite{Majda2002} and Danchin \cite{Danaxi}. The $t^{2}$-upper bound was established by Childress \cite{Child07}, and this can be derived by using an estimate for the velocity maximum from Feng--Sverak \cite{FeSv}. Recently, the $t^{4/3}$-upper bound was obtained by Jeong and the author \cite{LJ_optimal} for compactly supported initial vorticity with $\omg_{0}^{\tht}/r\in L^{\ift}$, and the compact support assumption was relaxed by Shao--Wei--Zhang \cite{SWZ25} and Egamberganov--Yao \cite{EYao25}. In addition, for single-signed vorticity, Maffei--Marchioro \cite{Maffei2001} provides $t^{1/4}\ln(e+t)$-upper bound using the conservation of the fluid impulse $\nrm{r\omg^{\tht}}_{L^{1}}$. For lower bounds, Choi--Jeong established a patch-type solution satisfying \eqref{eq:antiparallel} and $t^{1/15-}$-lower bound. This was enhanced by Gustafson--Miller--Tsai \cite{GMT2023} obtaining $t^{3/8-}$-lower bound. Also, \cite{EYao25} recently improved the lower bound up to $ t^{1/2-} $.
	

	
	
	\subsubsection{High dimension}
	
	There are several works which analyzed axisymmetric Euler equations in high dimensions $d\geq3$ (\cite{CJLglobal22, Limglobal23, GMT2023}). Although such cases are not physical, there are several works suggesting that there is a good chance that finite-time singularity formation of a solution with some smooth initial data can occur when the dimension is large (\cite{DE, Miller, HouZhang22P1, HouZhang22P2}). Recently, Miller \cite{Miller26prep} proved that when $d=4$, the global wellposedness holds for the initial data that satisfies $w_{0}/r^{2}\in L^{2,1}(\bbR^{4})$.
	
	The bi-rotational equation can be considered as a special case of \textit{the lake equation} with the depth function $d(r,s)=r^{n}s^{m}$ for any $n,m\geq1$. This appears in the work of Khesin--Yang \cite{KhYa} and Yang \cite{Yang}. Previously, Choi, Jeong, and the author \cite{CJL_gwpbi} proved the global wellposedness of the bi-rotational solutions with the Yudovich-type regularity and decay near the axes $r=0$, $s=0$, and at infinity for the 4D case $n=m=1$. In particular, they showed that $\nrm{w}_{L^{\ift}}$ has an exponential upper bound in time. As a recent work, Jeong and the author \cite{JL25} proved that for patch-type vorticity, the support diameter grows infinitely large.
	
	
	
	
	
	
	\subsection{Key ideas}
	
	\subsubsection{The extension of Childress' dipole model}
	
	\begin{figure}
		\centering
		\includegraphics[scale=0.3]{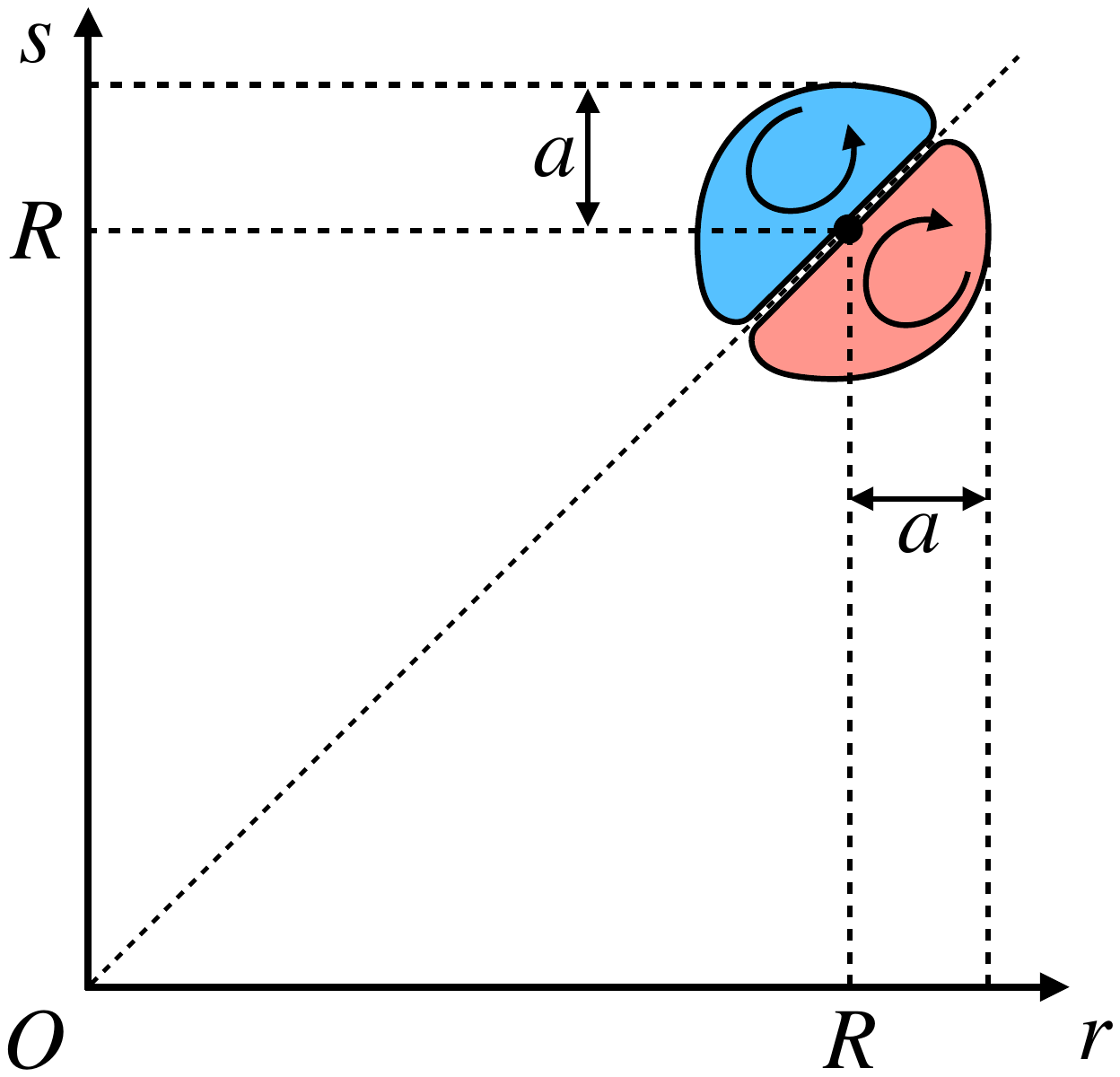}
		\caption{A diagram of Childress' dipole model centered at the point $(R,R)$ with radius $a$.}
		\label{fig:dipole}
	\end{figure}
	
	We begin by considering Childress' dipole model(\cite{Child08, ChilGil1}). This is a circular vortex dipole on the first quadrant $\{ (r,s)\in\bbR^{2} : r,s\geq0\}$ which is centered on the point $(R,R)$ for a given $R>0$ and a radius $a\ll R$. Additionally, $w$ is odd symmetric with respect to the axis $\{r=s\}$ (i.e., $w(s,r)=-w(r,s)$) and non-positive on the region $\{r\leq s \}$ (i.e., $w(r,s)\leq 0$ when $r\leq s$). This is depicted in Figure \ref{fig:dipole}. Then note that the velocity maximum is achieved on the center $(R,R)$ and the kinetic energy is scaled as
	\begin{equation}
		\nrm{u}_{L^{2}}^{2}\sim (R^{d-2}a)^{2}(R^{d-2}a^{2})\sim R^{3d-6}a^{4}.
	\end{equation} Then using the conservation of the kinetic energy, the radius $a$ is scaled as $a\sim R^{-(3d-6)/4}$. Then using that the vorticity maximum $w$ is scaled as $w\sim R^{d-2}$ and the velocity maximum is scaled as $u\sim wa\sim R^{d-2}a$, we obtain the following ODE:
	\begin{equation}
		\dot{R}\sim u\sim R^{d-2}a\sim R^{(d-2)/4},
	\end{equation}
	which gives us
	\begin{equation}
		R\sim \begin{cases}
			1+t^{2},&\quad\text{if}\quad d=4,\\
			1+t^{4},&\quad\text{if}\quad d=5,\\
			e^{ct},&\quad\text{if}\quad d=6.\\
		\end{cases}
	\end{equation}
	This leads to the rate
	\begin{equation}
		w\sim R^{d-2}\sim \begin{cases}
			1+t^{4},&\quad\text{if}\quad d=4,\\
			1+t^{12},&\quad\text{if}\quad d=5,\\
			e^{ct},&\quad\text{if}\quad d=6.\\
		\end{cases}
	\end{equation}
	Such structure is similar with the anti-parallel vortex rings on the half plane $\{r\geq0, z\in \bbR\}$. In fact, this rate coincides with the upper bound for axisymmetric solutions without swirl in $\bbR^{d}$, which was obtained in \cite{SWZ25}.
	
	\subsubsection{Conservation of the relative vorticity $w/(r^{n}s^{m})$ and the vorticity maximum}
	
	The first idea is to use the fact that the relative vorticity $w/(r^{n}s^{m})$ is conserved along the flow:
	\begin{equation}
		\frac{|w(t,\Phi_{t}^{r}(r,s),\Phi_{t}^{s}(r,s))|}{[\Phi_{t}^{r}(r,s)]^{n}[\Phi_{t}^{s}(r,s)]^{m}}=\frac{|w_{0}(r,s)|}{r^{n}s^{m}},
	\end{equation}
	where $\Phi_{t}=(\Phi_{t}^{r},\Phi_{t}^{s})$ is the flow map that uniquely satisfies the ODE
	\begin{equation}
		\begin{aligned}
			\frac{d}{dt}\Phi_{t}^{r}(r,s)=u^{r}(t,\Phi_{t}^{r}(r,s),\Phi_{t}^{s}(r,s)),\quad \Phi_{0}^{r}(r,s)=r,\\
			\frac{d}{dt}\Phi_{t}^{s}(r,s)=u^{s}(t,\Phi_{t}^{r}(r,s),\Phi_{t}^{s}(r,s)),\quad \Phi_{0}^{s}(r,s)=s.
		\end{aligned}
	\end{equation}
	Such conservation allows us to estimate the vorticity maximum by the time integral of the velocity maximum:	
	\begin{equation}\label{eq:wtmax}
		\nrm{w(t)}_{L^{\ift}(\bbR^{d})}\leq C(n,m,w_{0})\left(1+\int_{0}^{t}\nrm{u(\tau)}_{L^{\ift}(\bbR^{d})}d\tau\right)^{n+m},
	\end{equation}
	where $C(n,m,w_{0})>0$ is a constant that depends only on the norms $\nrm{w_{0}}_{L^{\ift}(\bbR^{d})}$, $\nrm{w_{0}/r^{n}}_{L^{\ift}(\bbR^{d})} $, $\nrm{w_{0}/s^{m}}_{L^{\ift}(\bbR^{d})}$, and $\nrm{w_{0}/(r^{n}s^{m})}_{L^{\ift}(\bbR^{d})}$. This naturally leads us to consider an estimate of the velocity maximum.
	
	\subsubsection{The Biot--Savart law and the Feng--\v{S}ver\'{a}k-type estimate}
	
	Next, the velocity maximum is bouneded by the following Feng--\v{S}ver\'{a}k-type estimate \eqref{eq:FStype} from Proposition \ref{prop:FStypeu}:
	\begin{equation}
		\nrm{u}_{L^{\ift}(\bbR^{d})}\lesssim \left[\left\| \frac{r^{m}w}{s^{m}}\right\|_{L^{1}(\bbR^{d})}+\left\| \frac{s^{n}w}{r^{n}}\right\|_{L^{1}(\bbR^{d})}\right]^{1/2}\left\| \frac{w}{r^{n}s^{m}}\right\|_{L^{\ift}(\bbR^{d})}^{1/2}.
	\end{equation}
	We have bounded the $L^{\ift}$-norm of the velocity $u $ by the $L^{\ift}$-norm of the relative velocity $w/(r^{n}s^{m})$ and two radial moments $\nrm{r^{m}w/s^{m}}_{L^{1}}$ and $\nrm{s^{n}w/r^{n}}_{L^{1}}$. 
	Thanks to the scale-invariance of the Biot--Savart law, the range of the estimate narrows down to the unit circle on the first quadrant (Inequality \eqref{eq:FStypescale}):
	\begin{equation}
		\sup_{\substack{r,s\geq0,\\r^{2}+s^{2}=1}}|u^{r}(r,s)|, \sup_{\substack{r,s\geq0,\\r^{2}+s^{2}=1}}|u^{s}(r,s)|\lesssim \left[\left\| \frac{r^{m}w}{s^{m}}\right\|_{L^{1}(\bbR^{d})}+\left\| \frac{s^{n}w}{r^{n}}\right\|_{L^{1}(\bbR^{d})}\right]^{1/2}\left\| \frac{w}{r^{n}s^{m}}\right\|_{L^{\ift}(\bbR^{d})}^{1/2}.
	\end{equation}
	In the previous work of the second author with Choi and Jeong(\cite{CJL_gwpbi}), we obtained an exponential upper bound of the vorticity maximum. Back then, we weren't able to find a way to simplify the bi-rotational Biot--Savart kernel similarly as in the axisymmetric case (\cite{FeSv}). But now, we managed to express the kernel by using a single variable function \textit{twice,} in some sense. This helps us in obtaining appropriate decay rates of the kernel. Due to this, the derivation of the estimate \eqref{eq:FStypescale} on the first quadrang becomes much simpler than the proof of the corresponding estimate from \cite{CJL_gwpbi}.
	
	\subsubsection{Radial moments}
	
	The last key idea is to estimate the time derivative of the radial moments $\nrm{r^{m}w/s^{m}}_{L^{1}}$ and $\nrm{s^{n}w/r^{n}}_{L^{1}}$ in Lemma \ref{lem:ddtPd-2}:
	\begin{align}
		\frac{d}{dt}\left\| \frac{r^{m}w(t)}{s^{m}}\right\|_{L^{1}(\bbR^{d})}&\leq \left\| \frac{w_{0}}{r^{n}s^{m}}\right\|_{L^{\ift}(\bbR^{d})}^{1/(d-2)}\nrm{\bfv_{0}}_{L^{2}(\bbR^{d})}^{2/(d-2)}\nrm{u^{r}(t)}_{L^{\ift}(\bbR^{d})}^{(d-4)/(d-2)}
		\left\| \frac{r^{m}w(t)}{s^{m}}\right\|_{L^{1}(\bbR^{d})}^{(d-3)/(d-2)},\\
		\frac{d}{dt}\left\| \frac{s^{n}w(t)}{r^{n}}\right\|_{L^{1}(\bbR^{d})}&\leq \left\| \frac{w_{0}}{r^{n}s^{m}}\right\|_{L^{\ift}(\bbR^{d})}^{1/(d-2)}\nrm{\bfv_{0}}_{L^{2}(\bbR^{d})}^{2/(d-2)}\nrm{u^{s}(t)}_{L^{\ift}(\bbR^{d})}^{(d-4)/(d-2)}
		\left\| \frac{s^{n}w(t)}{r^{n}}\right\|_{L^{1}(\bbR^{d})}^{(d-3)/(d-2)}.
	\end{align}
	Then combining this with the Feng--\v{S}ver\'{a}k-type estimate \eqref{eq:FStype}, we can obtain the differential inequality of the radial moments in Lemma \ref{lem:Ld-2time}:
	\begin{equation}
		\begin{aligned}
			\frac{d}{dt}&\left[\left\| \frac{r^{m}w(t)}{s^{m}}\right\|_{L^{1}(\bbR^{d})}+ \left\| \frac{s^{n}w(t)}{r^{n}}\right\|_{L^{1}(\bbR^{d})}\right]\lesssim_{n,m,w_{0}}
			\left[\left\| \frac{r^{m}w(t)}{s^{m}}\right\|_{L^{1}(\bbR^{d})}+ \left\| \frac{s^{n}w(t)}{r^{n}}\right\|_{L^{1}(\bbR^{d})}\right]^{(3d-10)/(2d-4)}.
		\end{aligned}
	\end{equation}
	Solving this differential inequality and plugging the rate obtained back into the estimate \eqref{eq:FStype}, we can get the rate of growth of $\nrm{u}_{L^{\ift}}$. Finally, inserting this rate in the estimate \eqref{eq:wtmax} finishes the proof of Theorem \ref{thm:GWP}.
	
	\subsection{Organization}
	In Section \ref{sec:BSlaw}, we recall the derivation of the Biot--Savart law for bi-rotational flows without swirl. In Section \ref{sec:lwp}, we establish the local wellposedness of Yudovich-type solutions for any $n,m\geq1$ by showing that the Biot--Savart kernel is a function in the dual space of $L^{d,1}(\bbR^{d})$. In Section \ref{sec:gwp}, we prove the global wellposedness of the local-in-time solutions for $d\leq6$ from the previous section with additional conditions assuming finite kinetic energy and some decays near the rotationa axes and at infinity. 

\section{The Biot--Savart law}\label{sec:BSlaw}
	
In this section, we present the derivation of the Biot--Savart formula under bi-rotational symmetry and no-swirl condition \eqref{eq:birothighd}. 
	
In $\bbR^{d}$, the Biot--Savart law is given as
\begin{equation}\label{eq:BSlawCartesian}
	v^{i}=\sum_{j=1}^{d}\rd_{x_{j}}\psi^{i,j}=\sum_{j=1}^{d}\rd_{j}\lap^{-1}\omg^{i,j}=\sum_{j=1}^{d}\lap^{-1}\rd_{j}\omg^{i,j}.
\end{equation}
We refer the readers to the book by Chemin \cite{Chemin} for details. We denote
\begin{equation}
\begin{aligned}
\Pi&:=\{(r,s)\in \bbR^{2} : r,s\geq0\},\\
X_{\pm\pm}&:=(r-\br{r})^{2}+2r\br{r}(1\pm\cos\br{\tht}_{1})+(s-\br{s})^{2}+2s\br{s}(1\pm\cos\br{\phi}_{1}),\\
X_{\pm\mp}&:=(r-\br{r})^{2}+2r\br{r}(1\pm\cos\br{\tht}_{1})+(s-\br{s})^{2}+2s\br{s}(1\mp\cos\br{\phi}_{1}).
\end{aligned}
\end{equation}
The explicit Biot--Savart formula is introduced in the following proposition.
	\begin{prop}\label{prop:bsformula}
		Let $(r,s)\in \Pi $. Then we have
		\begin{align}
			u^{r}(r,s)&=-
			\iint_{\Pi}F_{n,m}^{r}(r,s,\br{r},\br{s})w(\br{r},\br{s})
			d\br{r}d\br{s},\\
			F_{n,m}^{r}(r,s,\br{r},\br{s})&=c_{d}\int_{0}^{\pi}\int_{0}^{\pi}
			\frac{\br{r}^{n}\br{s}^{m}\sin^{n-1}\br{\tht}_{1}\cos\br{\tht}_{1}\sin^{m-1}\br{\phi}_{1}(\br{s}-s\cos\br{\phi}_{1})}{X_{--}^{d/2}}d\br{\tht}_{1}
			d\br{\phi}_{1},
		\end{align}
		and
		\begin{align}
			u^{s}(r,s)&=
			\iint_{\Pi}F_{n,m}^{s}(r,s,\br{r},\br{s})w(\br{r},\br{s})
			d\br{r}d\br{s},\\
			F_{n,m}^{s}(r,s,\br{r},\br{s})&=c_{d}\int_{0}^{\pi}\int_{0}^{\pi}
			\frac{\br{r}^{n}\br{s}^{m}\sin^{n-1}\br{\tht}_{1}(\br{r}-r\cos\br{\tht}_{1})\sin^{m-1}\br{\phi}_{1}\cos\br{\phi}_{1}}{X_{--}^{d/2}}d\br{\tht}_{1}
			d\br{\phi}_{1}.
		\end{align}
	\end{prop}

	\begin{proof}[Proof of Proposition \ref{prop:bsformula}]
		First, we use the change of coordinates to obtain the relation (cf. \cite{JL25})
		\begin{equation}
			\begin{aligned}
				\frac{u^{r}}{r}=\frac{v^{1}}{x_{1}}=-\frac{1}{x_{1}}\lap^{-1}\bigg[\frac{x_{1}}{r}
				\bigg(\rd_{s}w+\frac{m}{s}w\bigg)\bigg].
			\end{aligned}
		\end{equation}
		Computing this, we have
		\begin{equation}
			\begin{aligned}
				\bigg(\frac{v^{1}}{x_{1}}\bigg)(\bfx)&=-\frac{c_{d}}{x_{1}}\int_{\bbR^{d}}\frac{1}{|\bfx-\br{\bfx}|^{d-2}}\frac{\br{x}_{1}}{\br{r}}\bigg(\rd_{\br{s}}w(\br{\bfx})+\frac{m}{\br{s}}w(\br{\bfx})\bigg)d\br{\bfx}\\
				&=-\frac{c_{d}}{r\cos\tht_{1}}\iint_{\Pi}\int_{0}^{2\pi}
				\int_{[0,\pi]^{m-1}}\int_{0}^{2\pi}
				\int_{[0,\pi]^{n-1}}
				\frac{\br{r}^{n}\br{s}^{m}\cos\br{\tht}_{1}
				}{|x-\br{x}|^{d-2}}\bigg(\rd_{\br{s}}w(\br{r},\br{s})+\frac{m}{\br{s}}w(\br{r},\br{s})\bigg)\\
				&\qquad\cdot\bigg(\br{r}^{n}\br{s}^{m}\prod_{i=1}^{n-1}\sin^{n-i}\br{\tht}_{i}\prod_{j=1}^{m-1}\sin^{m-j}\br{\phi}_{j}\bigg)d\br{\tht}_{1}\cdots d\br{\tht}_{n-1}d\br{\tht}_{n}d\br{\phi}_{1}\cdots d\br{\phi}_{m-1}d\br{\phi}_{m}d\br{r}d\br{s}.
			\end{aligned}
		\end{equation}
		Then due to the rotational invariance of the Laplacian operator $\lap_{\bbR^{n+3}\times\bbR^{m+3}}$ and relations \eqref{eq:psicurl} and \eqref{eq:psiinvlap}, we have
		\begin{equation}
			\begin{aligned}
				\bigg(&\frac{v^{1}}{x_{1}}\bigg)(\bfx)=\bigg(\frac{v^{1}}{x_{1}}\bigg)(r\bfe_{r}+s\bfe_{s})
				\\
				&=-\frac{c_{d}}{r}\iint_{\Pi}\int_{0}^{\pi}\int_{0}^{\pi}\frac{\br{r}^{n}\br{s}^{m}\sin^{n-1}\br{\tht}_{1}\cos\br{\tht}_{1}\sin^{m-1}\br{\phi}_{1}
				}{X_{--}^{d/2-1}}\bigg(\rd_{\br{s}}w(\br{r},\br{s})+\frac{m}{\br{s}}w(\br{r},\br{s})\bigg)d\br{\tht}_{1}d\br{\phi}_{1}d\br{r}d\br{s}\\
				&=-\frac{c_{d}}{r}\iint_{\Pi}\int_{0}^{\pi}\int_{0}^{\pi}\sin^{n-1}\br{\tht}_{1}\cos\br{\tht}_{1}\sin^{m-1}\br{\phi}_{1}
				\bigg[-\rd_{\br{s}}\bigg(\frac{\br{r}^{n}\br{s}^{m}}{X_{--}^{d/2-1}}\bigg)+\frac{m\br{r}^{n}\br{s}^{m-1}}{X_{--}^{d/2-1}}\bigg]w(\br{r},\br{s})d\br{\tht}_{1}d\br{\phi}_{1}d\br{r}d\br{s}\\
				&=-\frac{c_{d}}{r}\iint_{\Pi}\int_{0}^{\pi}\int_{0}^{\pi}\frac{\br{r}^{n}\br{s}^{m}\sin^{n-1}\br{\tht}_{1}\cos\br{\tht}_{1}\sin^{m-1}\br{\phi}_{1}
					(\br{s}-s\cos\br{\phi}_{1})}{X_{--}^{d/2}}w(\br{r},\br{s})d\br{\tht}_{1}d\br{\phi}_{1}d\br{r}d\br{s}.
			\end{aligned}
		\end{equation}
		Thus, we obtain
		\begin{equation}
			\begin{aligned}
				u^{r}(r,s)&=r\bigg(\frac{v^{1}}{x_{1}}\bigg)(r\bfe_{r}+s\bfe_{s})
				\\
				&=-c_{d}\iint_{\Pi}\int_{0}^{\pi}\int_{0}^{\pi}\frac{\br{r}^{n}\br{s}^{m}\sin^{n-1}\br{\tht}_{1}\cos\br{\tht}_{1}\sin^{m-1}\br{\phi}_{1}
					(\br{s}-s\cos\br{\phi}_{1})}{X_{--}^{d/2}}w(\br{r},\br{s})d\br{\tht}_{1}d\br{\phi}_{1}d\br{r}d\br{s}.
			\end{aligned}
		\end{equation}
	We can derive $u^{s}$ in the same way by using the equation
		\begin{equation}
			\begin{aligned}
				\frac{u^{s}}{s}=\frac{v^{n+2}}{x_{n+2}}=\frac{1}{x_{n+2}}\lap^{-1}\bigg[\frac{x_{n+2}}{s}
				\bigg(\rd_{r}w+\frac{n}{r}w\bigg)\bigg].
			\end{aligned}
			\end{equation}
	\end{proof}

	\section{Local wellposedness of the Yudovich-type solution}\label{sec:lwp}
In this section, we show that the Yudovich-type solution of \eqref{eq:vorteq}, where the initial data $w_{0}$ is in $(L^{d,1}\cap L^{\ift})(\bbR^{d})$ and other terms $ w_{0}/(r^{n}s^{m}) $, $w_{0}/r$, and $w_{0}/s$ are in $L^{d,1}(\bbR^{d})$, is locally wellposed.
	
The following lemma shows that $L^{\ift}$-norms of the velocity with or without some weight is bounded by $L^{d,1}$-norms of the vorticity with or without the same weight.
	\begin{lem}\label{lem:lwplem}
		We have
		\begin{align}
			\nrm{\bfv}_{L^{\ift}(\bbR^{d})}
			&\leq C_{d}\nrm{w}_{L^{d,1}(\bbR^{d})},\label{eq:uwLd,1}\\
			\bigg\| \frac{u^{r}}{r}\bigg\|_{L^{\ift}(\bbR^{d})}&\leq C_{d}\bigg\| \frac{w}{r}\bigg\|_{L^{d,1}(\bbR^{d})},\label{eq:uroverrwoverrLd,1}\\
			\bigg\| \frac{u^{s}}{s}\bigg\|_{L^{\ift}(\bbR^{d})}&\leq C_{d}\bigg\| \frac{w}{s}\bigg\|_{L^{d,1}(\bbR^{d})}.\label{eq:usoverswoversLd,1}
		\end{align}
	\end{lem}
	
	\begin{proof}[Proof of Lemma \ref{lem:lwplem}]
		To prove \eqref{eq:uwLd,1}, we first use
		\begin{equation}
			(\br{s}-s\cos\br{\phi}_{1}
			)^{2}\leq X_{--}
		\end{equation}
		to get that
		\begin{equation}
			\begin{aligned}
				|F_{n,m}^{r}(r,s,\br{r},\br{s})|&\lesssim_{d} \int_{0}^{\pi}\int_{0}^{\pi}\frac{\br{r}^{n}\br{s}^{m}\sin^{n-1}\br{\tht}_{1}\sin^{m-1}\br{\phi}_{1}}{X_{--}^{(d-1)/2}}d\br{\tht}_{1}d\br{\phi}_{1}.
			\end{aligned}
		\end{equation}
		This gives us
\begin{equation}
	\begin{aligned}
		|u^{r}(r,s)|&\lesssim_{d} \iint_{\Pi}\int_{0}^{\pi}\int_{0}^{\pi}\frac{\br{r}^{n}\br{s}^{m}\sin^{n-1}\br{\tht}_{1}\sin^{m-1}\br{\phi}_{1}}{X_{--}^{(d-1)/2}}|w(\br{r},\br{s})|d\br{\tht}_{1}d\br{\phi}_{1}d\br{r}d\br{s}\\
		&\simeq_{d} \int_{\bbR^{d}}\frac{1}{|\bfx-\br{\bfx}|^{d-1}}|w(\br{\bfx})|d\br{\bfx}\lesssim_{d} \nrm{w}_{L^{d,1}(\bbR^{d})},
	\end{aligned}
\end{equation}
where the last line follows from Hölder's inequality in Lorentz spaces, since $|\bfx|^{-(d-1)}\in L^{d/(d-1),\infty}(\bbR^{d})$, the dual space of $L^{d,1}(\bbR^{d})$. In the same way, we can also prove
\begin{equation}
|u^{s}(r,s)|\lesssim_{d}\nrm{w}_{L^{d,1}(\bbR^{d})}.
\end{equation}

Now to show \eqref{eq:uroverrwoverrLd,1}, observe that we have
\begin{equation}
\begin{aligned}
F_{n,m}^{r}(r,s,\br{r},\br{s})
&=c_{d}\int_{0}^{\pi}\int_{0}^{\pi/2}
\br{r}^{n}\br{s}^{m}\sin^{n-1}\br{\tht}_{1}\cos\br{\tht}_{1}\sin^{m-1}\br{\phi}_{1}
(\br{s}-s\cos\br{\phi}_{1})\bigg(\frac{1}{X_{--}^{d/2}}-\frac{1}{X_{+-}^{d/2}}\bigg)d\br{\tht}_{1}d\br{\phi}_{1}.
\end{aligned}
\end{equation}
Since
\begin{equation}
\begin{aligned}
\frac{1}{X_{--}^{d/2}}-\frac{1}{X_{+-}^{d/2}} &= \frac{X_{+-}^{d}-X_{--}^{d}}{X_{--}^{d/2}X_{+-}^{d/2}(X_{+-}^{d/2}+X_{--}^{d/2})}\\
&=\frac{(X_{+-}-X_{--})(X_{+-}^{d-1}+X_{+-}^{d-2}X_{--}+\cdots+X_{+-}X_{--}^{d-2}+X_{--}^{d-1})}{X_{--}^{d/2}X_{+-}^{d/2}(X_{+-}^{d/2}+X_{--}^{d/2})}\\
&=\frac{4r\br{r}\cos\br{\tht}_{1}(X_{+-}^{d-1}+X_{+-}^{d-2}X_{--}+\cdots+X_{+-}X_{--}^{d-2}+X_{--}^{d-1})}{X_{--}^{d/2}X_{+-}^{d/2}(X_{+-}^{d/2}+X_{--}^{d/2})} \\
&\lesssim_{d} \frac{r\br{r}\cos\br{\tht}_{1}X_{+-}^{d-1}}{X_{--}^{d/2}X_{+-}^{d}}=\frac{r\br{r}\cos\br{\tht}_{1}}{X_{--}^{d/2}X_{+-}},
\end{aligned}
\end{equation}
we deduce that
\begin{equation}
\begin{aligned}
|F_{n,m}^{r}(r,s,\br{r},\br{s})|&\lesssim_{d} \int_{0}^{\pi}\int_{0}^{\pi/2}
\frac{r\br{r}^{n+1}\br{s}^{m}\sin^{n-1}\br{\tht}_{1}\cos^{2}\br{\tht}_{1}\sin^{m-1}\br{\phi}_{1}
|\br{s}-s\cos\br{\phi}_{1}|}{X_{--}^{d/2}X_{+-}}d\br{\tht}_{1}d\br{\phi}_{1}\\
&\leq \int_{0}^{\pi}\int_{0}^{\pi/2}
\frac{r\br{r}^{n+1}\br{s}^{m} \sin^{n-1}\br{\tht}_{1}\sin^{m-1}\br{\phi}_{1}}{X_{--}^{(d-1)/2}X_{+-}}d\br{\tht}_{1}d\br{\phi}_{1}.
\end{aligned}
\end{equation}
Next we use the inequality
\begin{equation}
\br{r}^{2}\leq 4X_{+-}.
\end{equation}
Indeed, this holds because if $r^{2}\leq X_{--}$, then $$ \br{r}=(\br{r}-r)+r\leq 2X_{--}^{1/2}, $$ and if $r^{2}> X_{--}$, then $ \br{r}\leq 2r\cos\br{\tht}_{1} $, which gives us $$ \br{r}^{2}\leq 2r\br{r}\cos\br{\tht}_{1}\leq X_{+-}. $$
So, using this inequality, we obtain
\begin{equation}
\begin{aligned}
|F_{n,m}^{r}(r,s,\br{r},\br{s})|&\lesssim_{d}  \int_{0}^{\pi}\int_{0}^{\pi/2}
				\frac{r\br{r}^{n-1}\br{s}^{m}\sin^{n-1}\br{\tht}_{1}\sin^{m-1}\br{\phi}_{1}}{X_{--}^{(d-1)/2}}d\br{\tht}_{1}d\br{\phi}_{1}\\
&\leq \int_{0}^{\pi}\int_{0}^{\pi}\frac{r\br{r}^{n-1}\br{s}^{m}\sin^{n-1}\br{\tht}_{1}\sin^{m-1}\br{\phi}_{1}}{X_{--}^{(d-1)/2}}d\br{\tht}_{1}d\br{\phi}_{1}.
\end{aligned}
\end{equation}
Thus, we arrive at 
\begin{equation}
\begin{aligned}
\bigg|\frac{1}{r}u^{r}(r,s)\bigg|&\leq \frac{1}{r}\iint_{\Pi}|F_{n,m}^{r}(r,s,\br{r},\br{s})||w(\br{r},\br{s})|d\br{r}d\br{s}\\
&\lesssim_{d} \iint_{\Pi}\int_{0}^{\pi}\int_{0}^{\pi}\frac{\br{r}^{n}\br{s}^{m}\sin^{n-1}\br{\tht}_{1}\sin^{m-1}\br{\phi}_{1}}{X_{--}^{(d-1)/2}}\frac{|w(\br{r},\br{s})|}{\br{r}}d\br{\tht}_{1}d\br{\phi}_{1}d\br{r}d\br{s}\\
&\simeq_{d} \int_{\bbR^{d}}\frac{1}{|\bfx-\br{\bfx}|^{d-1}}\frac{|w(\br{\bfx})|}{\br{r}}d\br{\bfx}\lesssim_{d} \bigg\| \frac{w}{r}\bigg\|_{L^{d,1}(\bbR^{d})}.
\end{aligned}
\end{equation}
The estimate \eqref{eq:usoverswoversLd,1} follows by analogous computations.
	\end{proof}
	
Now we prove Theorem \ref{thm:LWP}.
\begin{proof}[Proof of Theorem \ref{thm:LWP}]
It suffices to derive a priori estimates, as the existence of solutions follows by a standard smoothing argument, while uniqueness in the stated class can be proved as in \cite{Danaxi}. Hence, the local well-posedness reduces to establishing a priori estimates.

Applying \eqref{eq:uroverrwoverrLd,1} and \eqref{eq:usoverswoversLd,1} to the equation \eqref{eq:vorteq}, we obtain
\begin{equation}
	\frac{d}{dt}\nrm{w}_{L^{\ift}}\leq \left( n\left\| \frac{u^{r}}{r}\right\|_{L^{\ift}}+m\left\| \frac{u^{s}}{s}\right\|_{L^{\ift}}\right)\nrm{w}_{L^{\ift}}\lesssim_{m,n} \left( \left\| \frac{w}{r}\right\|_{L^{d,1}}+\left\| \frac{w}{s}\right\|_{L^{d,1}}\right)\nrm{w}_{L^{\ift}}
\end{equation}
and	
\begin{equation}
	\frac{d}{dt}\nrm{w}_{L^{d,1}}\leq \left( n\left\| \frac{u^{r}}{r}\right\|_{L^{\ift}}+m\left\| \frac{u^{s}}{s}\right\|_{L^{\ift}}\right)\nrm{w}_{L^{d,1}}\lesssim_{m,n} \left( \left\| \frac{w}{r}\right\|_{L^{d,1}}+\left\| \frac{w}{s}\right\|_{L^{d,1}}\right)\nrm{w}_{L^{d,1}}.
\end{equation}
By Gr\"{o}nwall's inequality, both $\|w\|_{L^{\infty}}$ and $\|w\|_{L^{d,1}}$ remain finite on every finite time interval $[0,T]$, provided that 
\begin{equation} \label{eq:bkm_type}
	\int_0^T \left( \left\|\frac{w(t)}{r}\right\|_{L^{d,1}} + \left\|\frac{w(t)}{s}\right\|_{L^{d,1}} \right)\,dt < \infty. 
\end{equation}
Moreover, since $w/(r^ns^m)$ is conserved along particle trajectories, we also have
\begin{equation}
	\left\| \frac{w(t)}{r^{n}s^{m}}\right\|_{L^{d,1}}=\left\| \frac{w_{0}}{r^{n}s^{m}}\right\|_{L^{d,1}}.
\end{equation}
It therefore remains to study the evolution of $w/r$, $w/s$, and show \eqref{eq:bkm_type} at least for a short time.

Observe that, for any integers $k,l \ge 0$, the quantity $w/(r^ks^l)$ satisfies		
\begin{align}
	\left( \partial_t + u^r \partial_r + u^s \partial_s\right)\frac{w}{r^ks^l} = \left( (n-k)\frac{u^r}{r} + (m-l)\frac{u^s}{s} \right) \frac{w}{r^ks^l}.
\end{align}
In particular, for $w/r$ and $w/s$, we have	
\begin{align}
	\rd_{t}\frac{w}{r}+(u^{r}\rd_{r}+u^{s}\rd_{s})\frac{w}{r}&=\left( (n-1)\frac{u^r}{r} + m\frac{u^s}{s} \right) \frac{w}{r},\\
	\rd_{t}\frac{w}{s}+(u^{r}\rd_{r}+u^{s}\rd_{s})\frac{w}{s}&=\left( n\frac{u^r}{r} + (m-1)\frac{u^s}{s} \right) \frac{w}{s}.
\end{align}
Combining these equations with \eqref{eq:uroverrwoverrLd,1} and \eqref{eq:usoverswoversLd,1} yields 
\begin{align}
	\frac{d}{dt} \left\| \frac{w}{r}\right\|_{L^{d,1}}&\leq \left( (n-1) \left\| \frac{u^{r}}{r}\right\|_{L^{\ift}} + m\left\| \frac{u^{s}}{s}\right\|_{L^{\ift}} \right) \left\| \frac{w}{r}\right\|_{L^{d,1}}\lesssim_{m,n} \left( \left\| \frac{w}{r}\right\|_{L^{d,1}} + \left\| \frac{w}{s}\right\|_{L^{d,1}} \right) \left\| \frac{w}{r}\right\|_{L^{d,1}},\\
	\frac{d}{dt} \left\| \frac{w}{s}\right\|_{L^{d,1}}&\leq \left( n \left\| \frac{u^{r}}{r}\right\|_{L^{\ift}} + (m-1)\left\| \frac{u^{s}}{s}\right\|_{L^{\ift}} \right) \left\| \frac{w}{s}\right\|_{L^{d,1}}\lesssim_{m,n} \left( \left\| \frac{w}{r}\right\|_{L^{d,1}} + \left\| \frac{w}{s}\right\|_{L^{d,1}} \right) \left\| \frac{w}{s}\right\|_{L^{d,1}}.
\end{align}
Consequently,	we get	
\begin{equation}
	\frac{d}{dt}\left( \left\| \frac{w}{r}\right\|_{L^{d,1}}+\left\| \frac{w}{s}\right\|_{L^{d,1}}\right)\lesssim_{m,n}\left( \left\| \frac{w}{r}\right\|_{L^{d,1}}+\left\| \frac{w}{s}\right\|_{L^{d,1}}\right)^{2},
\end{equation}
ensuring the boundedness of $\|w/r\|_{L^{d,1}}+\|w/s\|_{L^{d,1}}$ for some time interval $[0,T]$, where $T$ depends only on $\|w_0/r\|_{L^{d,1}}+\|w_0/s\|_{L^{d,1}}$. This implies \eqref{eq:bkm_type} on $[0,T]$ and finishes our proof.

\end{proof}

	\section{Global well-posedness of the Yudovich-type solution when $d\leq 6$}\label{sec:gwp}
	In this section, we prove the global wellposedness of the local-in-time solution $w$ from the previous section with some additional decay conditions on the initial data $w_{0}$ when the dimension is less than or equal to $6$.
	
	\subsection{Global velocity bound}
	
	In this subsection, our goal is to prove a global bound of the velocity $\bfv$ of Feng--Sverak-type (\cite{FeSv}). To do this, we introduce some technical lemmas on the decay property of the Biot--Savart kernel. \\ First, the following lemma is about the decay rate of some elliptic integral, which is going to be used many times in the proof of the next lemma.
	\begin{lem}\label{lem:calAdecay}
		Let $\alp\in \{
		0,1\}$, $\bt\geq0$, $\gmm>(\bt+1)/2$, and define
		\begin{equation}
			\calA_{\alp,\bt,\gmm}(\tau):=\int_{0}^{\pi}\frac{\cos^{\alp}\tht\sin^{\bt}\tht}{[\tau+2(1-\cos\tht)]^{\gmm}}d\tht,\quad 
			\tau>0.
		\end{equation}
		Then for any $\tau>0$, we have
		\begin{equation}\label{eq:calAdecay}
			0\leq\calA_{\alp,\bt,\gmm}(\tau)\lesssim_{\alp,\bt,\gmm} \min\lbrace \tau^{-(\gmm-(\bt+1)/2)}, \tau^{-(\gmm+\alp)}\rbrace.
		\end{equation}
	\end{lem}
	
	\begin{proof}[Proof of Lemma \ref{lem:calAdecay}]
		The non-negativity of $\calA_{\alp,\bt,\gmm}(\tau)$ follows from
		\begin{equation}
			\calA_{\alp,\bt,\gmm}(\tau)=\int_{0}^{\pi/2}\left[ \frac{\cos^{\alp}\tht\sin^{\bt}\tht}{[\tau+2(1-\cos\tht)]^{\gmm}}+(-1)^{\alp}\frac{\cos^{\alp}\tht\sin^{\bt}\tht}{[\tau+2(1+\cos\tht)]^{\gmm}}\right]d\tht\geq0.
		\end{equation}
		We want to show that we have
		\begin{align}
			\calA_{\alp,\bt,\gmm}(\tau)\lesssim_{\bt,\gmm}\tau^{-(\gmm-(\bt+1)/2)}
			&\quad\text{for any}
			\quad 0<\tau<1,
			\label{eq:calAslow}\\
			\calA_{\alp,\bt,\gmm}(\tau)\lesssim_{\alp,\gmm}\tau^{-(\gmm+\alp)}
			&\quad\text{for any}
			\quad \tau\geq1.
			\label{eq:calAfast}
		\end{align}
		To prove \eqref{eq:calAslow}, note that we have
		\begin{equation}
			\sin\tht=\tht+O(\tht^{2}),\quad 1-\cos\tht=\frac{\tht^{2}}{2}+O(\tht^{4})\quad\text{as}\quad \tht\to0^{+}.
		\end{equation}
		That is, we can choose a sufficiently small $\veps>0$ that satisfies
		\begin{equation}
			\sin\tht\leq \tht,\quad 1-\cos\tht\geq \frac{\tht^{2}}{2}\quad\text{for any}\quad \tht\in[0,\veps].
		\end{equation}
		Also, we can choose a constant $c>0$ satisfying
		\begin{equation}
			1-\cos\tht\geq \frac{c}{2}\quad\text{for any}\quad \tht\in [\veps,\pi].
		\end{equation}
		Then for any $0<\tau<1$, 
		we have
		\begin{equation}
			\begin{aligned}
				\calA_{\alp,\bt,\gmm}(\tau)&=\left[\int_{0}^{\veps}+\int_{\veps}^{\pi}\right]\frac{\cos^{\alp}\tht\sin^{\bt}\tht}{[\tau+2(1-\cos\tht)]^{\gmm}}d\tht\\
				&\leq \int_{0}^{\veps}\frac{\tht^{\bt}}{(\tau+\tht^{2})^{\gmm}}d\tht+\int_{\veps}^{\pi}\frac{\cos^{\alp}\tht\sin^{\bt}\tht}{[\tau+2(1-\cos\tht)]^{\gmm}}d\tht\\
				&=\int_{0}^{\veps/\sqrt{\tau}}\frac{\tau^{\bt/2}\phi^{\bt}}{\tau^{\gmm}(1+\phi^{2})^{\gmm}}\sqrt{\tau}d\phi+\frac{\tau^{\gmm-(\bt+1)/2}}{\tau^{\gmm-(\bt+1)/2}}\int_{\veps}^{\pi}\frac{\cos^{\alp}\tht\sin^{\bt}\tht}{[\tau+2(1-\cos\tht)]^{\gmm}}d\tht.
			\end{aligned}
		\end{equation}
		
		\begin{equation}
			\begin{aligned}
				\calA_{\alp,\bt,\gmm}(\tau)&\leq\frac{1}{\tau^{\gmm-(\bt+1)/2}}\int_{0}^{\ift}\frac{\phi^{\bt}}{(1+\phi^{2})^{\gmm}}d\phi+\frac{1}{\tau^{\gmm-(\bt+1)/2}}\int_{\veps}^{\pi}\frac{1}{c^{\gmm}}d\tht\\
				&\lesssim_{\bt,\gmm}\frac{1}{\tau^{\gmm-(\bt+1)/2}}.
			\end{aligned}
		\end{equation}
		Now to prove \eqref{eq:calAfast}, we define $z:=1/\tau$ and 
		\begin{equation}
			\begin{aligned}
				g(z)&:=\calA_{\alp,\bt,\gmm}(1/z)=\int_{0}^{\pi}\frac{\cos^{\alp}\tht\sin^{\bt}\tht}{[1/z+2(1-\cos\tht)]^{\gmm}}d\tht=z^{\gmm}\int_{0}^{\pi}\frac{\cos^{\alp}\tht\sin^{\bt}\tht}{[1+2(1-\cos\tht)z]^{\gmm}}d\tht
			\end{aligned}
		\end{equation}
		Then we use the expansion
		\begin{equation}
			\frac{1}{(1+a)^{\gmm}}=1-\gmm a+O(a^{2})\quad\text{as}\quad a\to 0^{+}
		\end{equation}
		to get
		\begin{equation}
			\begin{aligned}
				g(z)&=z^{\gmm}\left[ \int_{0}^{\pi}\cos^{\alp}\tht\sin^{\bt}\tht\tht-\gmm 2z\int_{0}^{\pi}\cos^{\alp}\tht\sin^{\bt}\tht(1-\cos\tht)d\tht+O(z^{2})\right]\\
				&=C_{\bt,\gmm}z^{\gmm+\alp}+O(z^{\gmm+\alp+1}),
			\end{aligned}
		\end{equation}
		as $z\to0^{+}$. Thus, we obtain \eqref{eq:calAfast}.
	\end{proof}

	Now we let $(r,s), (\br{r},\br{s})\in \Pi$ and denote $D^{2}:=(r-\br{r})^{2}+(s-\br{s})^{2}$. Then we can rewrite the kernels $ F_{n,m}^{r} $ and $F_{n,m}^{s}$ as
	\begin{equation}\label{eq:Fnmr}
		\begin{aligned}
			&F_{n,m}^{r}(r,s,\br{r},\br{s})=c_{d}\int_{0}^{\pi}\int_{0}^{\pi}
			\frac{\br{r}^{n}\br{s}^{m}\sin^{n-1}\br{\tht}_{1}\cos\br{\tht}_{1}\sin^{m-1}\br{\phi}_{1}(\br{s}-s\cos\br{\phi}_{1})}{X_{--}^{d/2}}d\br{\tht}_{1}
			d\br{\phi}_{1}\\
			\qquad&\simeq_{d}\int_{0}^{\pi}\br{r}^{n}\br{s}^{m}\sin^{m-1}\br{\phi}_{1}(\br{s}-s\cos\br{\phi_{1}})\int_{0}^{\pi}\frac{\sin^{n-1}\br{\tht}_{1}\cos\br{\tht_{1}}}{(r\br{r})^{d/2}\left[\frac{D^{2}+2s\br{s}(1-\cos\br{\phi}_{1})}{r\br{r}}+2(1-\cos\br{\tht}_{1})\right]^{d/2}}d\br{\tht}_{1}d\br{\phi}_{1}\\
			\qquad&=\int_{0}^{\pi}\frac{\br{r}^{n}\br{s}^{m}\sin^{m-1}\br{\phi}_{1}(\br{s}-s\cos\br{\phi_{1}})}{(r\br{r})^{d/2}}\calA_{1,n-1,d/2}\left(\frac{D^{2}+2s\br{s}(1-\cos\br{\phi}_{1})}{r\br{r}}\right)d\br{\phi}_{1},
		\end{aligned}
	\end{equation}
	and
	\begin{equation}\label{eq:Fnms}
		\begin{aligned}
			&F_{n,m}^{s}(r,s,\br{r},\br{s})=c_{d}\int_{0}^{\pi}\int_{0}^{\pi}
			\frac{\br{r}^{n}\br{s}^{m}\sin^{n-1}\br{\tht}_{1}(\br{r}-r\cos\br{\tht}_{1})\sin^{m-1}\br{\phi}_{1}\cos\br{\phi}_{1}}{X_{--}^{d/2}}d\br{\tht}_{1}
			d\br{\phi}_{1}\\
			\qquad&\simeq_{d}\int_{0}^{\pi}\br{r}^{n}\br{s}^{m}\sin^{n-1}\br{\tht}_{1}(\br{r}-r\cos\br{\tht_{1}})\int_{0}^{\pi}\frac{\sin^{m-1}\br{\phi}_{1}\cos\br{\phi_{1}}}{(s\br{s})^{d/2}\left[\frac{D^{2}+2r\br{r}(1-\cos\br{\tht}_{1})}{s\br{s}}+2(1-\cos\br{\phi}_{1})\right]^{d/2}}d\br{\phi}_{1}d\br{\tht}_{1}\\
			\qquad&=\int_{0}^{\pi}\frac{\br{r}^{n}\br{s}^{m}\sin^{n-1}\br{\tht}_{1}(\br{r}-r\cos\br{\tht_{1}})}{(s\br{s})^{d/2}}\calA_{1,m-1,d/2}\left(\frac{D^{2}+2r\br{r}(1-\cos\br{\tht}_{1})}{s\br{s}}\right)d\br{\tht}_{1}.
		\end{aligned}
	\end{equation}
	Moreover, they satisfy the following decay property.
	\begin{lem}\label{lem:Frslow}
		For any $(r,s), (\br{r},\br{s})\in \Pi$, the kernels $F_{n,m}^{r}$ and $F_{n,m}^{s}$ satisfy
		\begin{equation}\label{eq:Frslow}
			|F_{n,m}^{r}(r,s,\br{r},\br{s})|, |F_{n,m}^{s}(r,s,\br{r},\br{s})|\lesssim_{n,m} \frac{1}{D}.
		\end{equation} 
	\end{lem}
	
	\begin{proof}[Proof of Lemma \ref{lem:Frslow}]
		Throughout the proof, we often use the following inequality
		\begin{equation}\label{eq:basic}
			r, \br{r}\lesssim |r-\br{r}|+(r\br{r})^{1/2},\quad s, \br{s}\lesssim |s-\br{s}|+(s\br{s})^{1/2},
		\end{equation}
		which is obtained by
		\begin{equation}
			\begin{aligned}
				r&\leq r^{1/2}[|r-\br{r}|+\br{r}]^{1/2}\lesssim r^{1/2}|r-\br{r}|^{1/2}+(r\br{r})^{1/2}\leq \frac{r}{2}+\frac{|r-\br{r}|}{2}+(r\br{r})^{1/2},\\
				\br{r}&\leq \br{r}^{1/2}[|r-\br{r}|+r]^{1/2}\lesssim \frac{\br{r}}{2}+\frac{|r-\br{r}|}{2}+(r\br{r})^{1/2}.
			\end{aligned}
		\end{equation}
		First, we use the inequality \eqref{eq:basic} to get
		\begin{equation}
			\begin{aligned}
				&|F_{n,m}^{r}(r,s,\br{r},\br{s})|\lesssim \int_{0}^{\pi}\frac{\br{r}^{n}\br{s}^{m}\sin^{m-1}\br{\phi}_{1}|\br{s}-s\cos\br{\phi_{1}}|}{(r\br{r})^{d/2}}
				\calA_{1,n-1,d/2}\left(\frac{D^{2}+2s\br{s}(1-\cos\br{\phi}_{1})}{r\br{r}}\right)d\br{\phi}_{1}\\
				\quad&=\int_{0}^{\pi}\frac{r\br{r}^{n+1}\br{s}^{m}\sin^{m-1}\br{\phi}_{1}|\br{s}-s\cos\br{\phi_{1}}|}{r^{d/2+1}\br{r}^{d/2+1}}\calA_{1,n-1,d/2}\left(\frac{D^{2}+2s\br{s}(1-\cos\br{\phi}_{1})}{r\br{r}}\right)d\br{\phi}_{1}\\
				\quad&\lesssim 
				\int_{0}^{\pi}\frac{[|r-\br{r}|^{n+2}+(r\br{r})^{n/2+1}]\br{s}^{m}\sin^{m-1}\br{\phi}_{1}|\br{s}-s\cos\br{\phi_{1}}|}{r^{d/2+1}\br{r}^{d/2+1}}\calA_{1,n-1,d/2}\left(\frac{D^{2}+2s\br{s}(1-\cos\br{\phi}_{1})}{r\br{r}}\right)d\br{\phi}_{1}.
			\end{aligned}
		\end{equation}
		Then using the decay property
		\begin{equation}
			\calA_{1,n-1,d/2}(\tau)\lesssim_{n,m} \min\{ \tau^{-(m/2+1)},\tau^{-(d/2+1)}\}
		\end{equation}
		of the integral $ \calA_{1,n-1,d/2} $ from \eqref{eq:calAdecay} in Lemma \ref{lem:calAdecay}, we have
		\begin{equation}
			\begin{aligned}
				&|F_{n,m}^{r}(r,s,\br{r},\br{s})|\\
				\quad&\lesssim _{n,m}
				\int_{0}^{\pi}\frac{\br{s}^{m}\sin^{m-1}\br{\phi}_{1}|\br{s}-s\cos\br{\phi_{1}}|}{r^{d/2+1}\br{r}^{d/2+1}}\\
				&\qquad\cdot\left[|r-\br{r}|^{n+2}\left( \frac{r\br{r}}{D^{2}+2s\br{s}(1-\cos\br{\phi}_{1})}\right)^{d/2+1}+(r\br{r})^{n/2+1} \left( \frac{r\br{r}}{D^{2}+2s\br{s}(1-\cos\br{\phi}_{1})}\right)^{m/2+1}\right]d\br{\phi}_{1}\\
				\quad&\lesssim \int_{0}^{\pi}\frac{
					\br{s}^{m}\sin^{m-1}\br{\phi}_{1}}{[D^{2}+2s\br{s}(1-\cos\br{\phi}_{1})]^{(m+1)/2}}d\br{\phi}_{1}=\frac{\br{s}^{m}}{(s\br{s})^{(m+1)/2}}\calA_{0,m-1,(m+1)/2}\left(\frac{D^{2}}{s\br{s}}\right).
			\end{aligned}
		\end{equation}
		Then we use the inequality \eqref{eq:basic} and the decay
		\begin{equation}
			\calA_{0,m-1,(m+1)/2}(\tau)\lesssim_{m} \min\{ \tau^{-(m+1)/2},\tau^{-1/2}\}
		\end{equation}
		from Lemma \ref{lem:calAdecay} to obtain
		\begin{equation}
			\begin{aligned}
				&|F_{n,m}^{r}(r,s,\br{r},\br{s})|\lesssim \frac{[|s-\br{s}|^{m}+(s\br{s})^{m/2}]}{(s\br{s})^{(m+1)/2}}\calA_{0,m-1,(m+1)/2}\left(\frac{D^{2}}{s\br{s}}\right)
				\\
				\quad&\lesssim_{m} \frac{
					1}{(s\br{s})^{(m+1)/2}}\left[|s-\br{s}|^{m}\left(\frac{s\br{s}}{D^{2}}\right)^{(m+1)/2}+(s\br{s})^{m/2}\left(\frac{s\br{s}}{D^{2}}\right)^{1/2}\right]\lesssim \frac{1
				}{D}.
			\end{aligned}
		\end{equation}
		Similarly, we use the inequality \eqref{eq:basic} to get
		\begin{equation}
			\begin{aligned}
				&|F_{n,m}^{s}(r,s,\br{r},\br{s})|\lesssim \int_{0}^{\pi}\frac{\br{r}^{n}\br{s}^{m}\sin^{n-1}\br{\tht}_{1}|\br{r}-r\cos\br{\tht_{1}}|}{(s\br{s})^{d/2}}
				\calA_{1,m-1,d/2}\left(\frac{D^{2}+2r\br{r}(1-\cos\br{\tht}_{1})}{s\br{s}}\right)d\br{\tht}_{1}\\
				\quad&=\int_{0}^{\pi}\frac{\br{r}^{n}s\br{s}^{m+1}\sin^{n-1}\br{\tht}_{1}|\br{r}-r\cos\br{\tht_{1}}|}{s^{d/2+1}\br{s}^{d/2+1}}\calA_{1,m-1,d/2}\left(\frac{D^{2}+2r\br{r}(1-\cos\br{\tht}_{1})}{s\br{s}}\right)d\br{\tht}_{1}\\
				\quad&\lesssim 
				\int_{0}^{\pi}\frac{\br{r}^{n}[|s-\br{s}|^{m+2}+(s\br{s})^{m/2+1}]\sin^{n-1}\br{\tht}_{1}|\br{r}-r\cos\br{\tht_{1}}|}{s^{d/2+1}\br{s}^{d/2+1}}\calA_{1,m-1,d/2}\left(\frac{D^{2}+2r\br{r}(1-\cos\br{\tht}_{1})}{s\br{s}}\right)d\br{\tht}_{1}.
			\end{aligned}
		\end{equation}
		Then we use the decay property
		\begin{equation}
			\calA_{1,m-1,d/2}(\tau)\lesssim_{n,m} \min\{ \tau^{-(n/2+1)},\tau^{-(d/2+1)}\},
		\end{equation}
		which gives us
		\begin{equation}
			\begin{aligned}
				&|F_{n,m}^{s}(r,s,\br{r},\br{s})|\\
				\quad&\lesssim 
				\int_{0}^{\pi}\frac{\br{r}^{n}\sin^{n-1}\br{\tht}_{1}|\br{r}-r\cos\br{\tht_{1}}|}{s^{d/2+1}\br{s}^{d/2+1}}\\
				&\qquad\cdot\left[|s-\br{s}|^{m+2}\left( \frac{s\br{s}}{D^{2}+2r\br{r}(1-\cos\br{\tht}_{1})}\right)^{d/2+1}+(s\br{s})^{m/2+1} \left( \frac{s\br{s}}{D^{2}+2r\br{r}(1-\cos\br{\tht}_{1})}\right)^{n/2+1}\right]d\br{\tht}_{1}\\
				\quad&\lesssim \int_{0}^{\pi}\frac{
					\br{r}^{n}\sin^{n-1}\br{\tht}_{1}}{[D^{2}+2r\br{r}(1-\cos\br{\tht}_{1})]^{(n+1)/2}}d\br{\tht}_{1}=\frac{\br{r}^{n}}{(r\br{r})^{(n+1)/2}}\calA_{0,n-1,(n+1)/2}\left(\frac{D^{2}}{r\br{r}}\right).
			\end{aligned}
		\end{equation}
		Finally, we use the decay
		\begin{equation}
			\calA_{0,n-1,(n+1)/2}(\tau)\lesssim_{n} \min\{ \tau^{-(n+1)/2},\tau^{-1/2}\}
		\end{equation}
		to obtain
		\begin{equation}
			\begin{aligned}
				&|F_{n,m}^{s}(r,s,\br{r},\br{s})|\lesssim 
				\frac{[|r-\br{r}|^{n}+(r\br{r})^{n/2}]}{(r\br{r})^{(n+1)/2}}\calA_{0,n-1,(n+1)/2}\left(\frac{D^{2}}{r\br{r}}\right)
				\\
				\quad&\lesssim_{n} \frac{
					1}{(r\br{r})^{(n+1)/2}}\left[|r-\br{r}|^{n}\left(\frac{r\br{r}}{D^{2}}\right)^{(n+1)/2}+(r\br{r})^{n/2}\left(\frac{r\br{r}}{D^{2}}\right)^{1/2}\right]\lesssim \frac{1
				}{D}.
			\end{aligned}
		\end{equation}
	\end{proof}
	Similarly, the kernels also have the following decay rate.
	\begin{lem}\label{lem:FrFsweight}
		For any $(r,s), (\br{r},\br{s})\in \Pi$, the kernels $F_{n,m}^{r}$ and $F_{n,m}^{s}$ satisfy
		\begin{align}
			|F_{n,m}^{r}(r,s,\br{r},\br{s})|\lesssim \frac{r}{D^{2}},\label{eq:Frweight}\\
			|F_{n,m}^{s}(r,s,\br{r},\br{s})|\lesssim \frac{s}{D^{2}}.\label{eq:Fsweight}
		\end{align}
	\end{lem}
	
	\begin{proof}[Proof of Lemma \ref{lem:FrFsweight}]
		We slightly change the proof of Lemma \ref{lem:Frslow}. To begin with, we use the elementary inequality \eqref{eq:basic};
		\begin{equation}
			r, \br{r}\lesssim |r-\br{r}|+(r\br{r})^{1/2},\quad s, \br{s}\lesssim |s-\br{s}|+(s\br{s})^{1/2},
		\end{equation}
		to get
		\begin{equation}
			\begin{aligned}
				&|F_{n,m}^{r}(r,s,\br{r},\br{s})|\lesssim 
				\int_{0}^{\pi}\frac{r\br{r}^{n+1}\br{s}^{m}\sin^{m-1}\br{\phi}_{1}|\br{s}-s\cos\br{\phi_{1}}|}{r^{d/2+1}\br{r}^{d/2+1}}\calA_{1,n-1,d/2}\left(\frac{D^{2}+2s\br{s}(1-\cos\br{\phi}_{1})}{r\br{r}}\right)d\br{\phi}_{1}\\
				\;&\lesssim 
				\int_{0}^{\pi}\frac{r[|r-\br{r}|^{n+1}+(r\br{r})^{(n+1)/2}]\br{s}^{m}\sin^{m-1}\br{\phi}_{1}|\br{s}-s\cos\br{\phi_{1}}|}{r^{d/2+1}\br{r}^{d/2+1}}\calA_{1,n-1,d/2}\left(\frac{D^{2}+2s\br{s}(1-\cos\br{\phi}_{1})}{r\br{r}}\right)d\br{\phi}_{1}.
			\end{aligned}
		\end{equation}
		In the above, note that we left the term $r$ on the numerator unchanged. Then we use the decay
		\begin{equation}
			\calA_{1,n-1,d/2}(\tau)\lesssim_{n,m} \min\{ \tau^{-(m/2+1)},\tau^{-(d/2+1)}\}.
		\end{equation}
		This gives us
		\begin{equation}
			\begin{aligned}
				&|F_{n,m}^{r}(r,s,\br{r},\br{s})|\\
				\;&\lesssim 
				\int_{0}^{\pi}\frac{r\br{s}^{m}\sin^{m-1}\br{\phi}_{1}|\br{s}-s\cos\br{\phi_{1}}|}{r^{d/2+1}\br{r}^{d/2+1}}\\
				&\quad\cdot\left[|r-\br{r}|^{n+1}\left( \frac{r\br{r}}{D^{2}+2s\br{s}(1-\cos\br{\phi}_{1})}\right)^{d/2+1}+(r\br{r})^{(n+1)/2} \left( \frac{r\br{r}}{D^{2}+2s\br{s}(1-\cos\br{\phi}_{1})}\right)^{(m+3)/2}\right]d\br{\phi}_{1}\\
				\quad&\lesssim \int_{0}^{\pi}\frac{
					r\br{s}^{m}\sin^{m-1}\br{\phi}_{1}}{[D^{2}+2s\br{s}(1-\cos\br{\phi}_{1})]^{m/2+1}}d\br{\phi}_{1}=\frac{r\br{s}^{m}}{(s\br{s})^{m/2+1}}\calA_{0,m-1,m/2+1}\left(\frac{D^{2}}{s\br{s}}\right).
			\end{aligned}
		\end{equation}
		Note that we used the decay rate $-(m+3)/2$ in the computation above. 
		Then we use the decay
		\begin{equation}
			\calA_{0,m-1,m/2+1}(\tau)\lesssim_{m} \min\{ \tau^{-(m/2+1)},\tau^{-1}\},
		\end{equation}
		which leads to
		\begin{equation}
			\begin{aligned}
				&|F_{n,m}^{r}(r,s,\br{r},\br{s})|\lesssim 
				\frac{r[|s-\br{s}|^{m}+(s\br{s})^{m/2}]}{(s\br{s})^{m/2+1}}\calA_{0,m-1,m/2+1}\left(\frac{D^{2}}{s\br{s}}\right)
				\\
				\quad&\lesssim \frac{
					r}{(s\br{s})^{m/2+1}}\left[|s-\br{s}|^{m}\left(\frac{s\br{s}}{D^{2}}\right)^{m/2+1}+(s\br{s})^{m/2}\left(\frac{s\br{s}}{D^{2}}\right)\right]\lesssim \frac{r
				}{D^{2}}.
			\end{aligned}
		\end{equation}
		In a similar way, we use the elementary inequality to get
		\begin{equation}
			\begin{aligned}
				&|F_{n,m}^{s}(r,s,\br{r},\br{s})|\lesssim 
				\int_{0}^{\pi}\frac{\br{r}^{n}s\br{s}^{m+1}\sin^{n-1}\br{\tht}_{1}|\br{r}-r\cos\br{\tht_{1}}|}{s^{d/2+1}\br{s}^{d/2+1}}\calA_{1,m-1,d/2}\left(\frac{D^{2}+2r\br{r}(1-\cos\br{\tht}_{1})}{s\br{s}}\right)d\br{\tht}_{1}\\
				\;&\lesssim 
				\int_{0}^{\pi}\frac{\br{r}^{n}s[|s-\br{s}|^{m+1}+(s\br{s})^{(m+1)/2}]\sin^{n-1}\br{\tht}_{1}|\br{r}-r\cos\br{\tht_{1}}|}{s^{d/2+1}\br{s}^{d/2+1}}\calA_{1,m-1,d/2}\left(\frac{D^{2}+2r\br{r}(1-\cos\br{\tht}_{1})}{s\br{s}}\right)d\br{\tht}_{1},
			\end{aligned}
		\end{equation}
		where the term$ s $ in the numerator above stays the same. Then using the decay
		\begin{equation}
			\calA_{1,m-1,d/2}(\tau)\lesssim_{n,m} \min\{ \tau^{-(n/2+1)},\tau^{-(d/2+1)}\},
		\end{equation}
		we have
		\begin{equation}
			\begin{aligned}
				&|F_{n,m}^{s}(r,s,\br{r},\br{s})|\\
				\;&\lesssim 
				\int_{0}^{\pi}\frac{\br{r}^{n}s\sin^{n-1}\br{\tht}_{1}|\br{r}-r\cos\br{\tht_{1}}|}{s^{d/2+1}\br{s}^{d/2+1}}\\
				&\quad\cdot\left[|s-\br{s}|^{m+1}\left( \frac{s\br{s}}{D^{2}+2r\br{r}(1-\cos\br{\tht}_{1})}\right)^{d/2+1}+(s\br{s})^{(m+1)/2} \left( \frac{s\br{s}}{D^{2}+2r\br{r}(1-\cos\br{\tht}_{1})}\right)^{(n+3)/2}\right]d\br{\tht}_{1}\\
				\;&\lesssim \int_{0}^{\pi}\frac{
					\br{r}^{n}s\sin^{n-1}\br{\tht}_{1}}{[D^{2}+2r\br{r}(1-\cos\br{\tht}_{1})]^{n/2+1}}d\br{\tht}_{1}=\frac{\br{r}^{n}s}{(r\br{r})^{n/2+1}}\calA_{0,n-1,n/2+1}\left(\frac{D^{2}}{r\br{r}}\right).
			\end{aligned}
		\end{equation}
		In the above, we used the decay rate $-(n+3)/2$. Then we use the decay
		\begin{equation}
			\calA_{0,n-1,n/2+1}(\tau)\lesssim_{n} \min\{ \tau^{-(n/2+1)},\tau^{-1}\},
		\end{equation}
		which gives us
		\begin{equation}
			\begin{aligned}
				&|F_{n,m}^{s}(r,s,\br{r},\br{s})|\lesssim 
				\frac{[|r-\br{r}|^{n}+(r\br{r})^{n/2}]s}{(r\br{r})^{n/2+1}}\calA_{0,n-1,n/2+1}\left(\frac{D^{2}}{r\br{r}}\right)
				\\
				\quad&\lesssim \frac{
					s}{(r\br{r})^{n/2+1}}\left[|r-\br{r}|^{n}\left(\frac{r\br{r}}{D^{2}}\right)^{n/2+1}+(r\br{r})^{n/2}\left(\frac{r\br{r}}{D^{2}}\right)\right]\lesssim \frac{s
				}{D^{2}}.
			\end{aligned}
		\end{equation}
	\end{proof}
	Now we introduce a proposition about time-independent global bound of the velocity $\bfv$. This is going to be one of the main idea in the proof of Theorem \ref{thm:GWP}.
	\begin{prop}\label{prop:FStypeu}
		We have
		\begin{equation}\label{eq:FStype}
			\nrm{\bfv}_{L^{\ift}(\bbR^{d})}\lesssim \left[\left\| \frac{r^{m}w}{s^{m}}\right\|_{L^{1}(\bbR^{d})}+\left\| \frac{s^{n}w}{r^{n}}\right\|_{L^{1}(\bbR^{d})}\right]^{1/2}\left\| \frac{w}{r^{n}s^{m}}\right\|_{L^{\ift}(\bbR^{d})}^{1/2}.
		\end{equation}
	\end{prop}
	
	\begin{proof}[Proof of Proposition \ref{prop:FStypeu}]
		It suffices to show that for any $r,s\geq0$ satisfying $r^{2}+s^{2}=1$, we have
		\begin{equation}\label{eq:FStypescale}
			|u^{r}(r,s)|, |u^{s}(r,s)|\lesssim \left[\left\| \frac{r^{m}w}{s^{m}}\right\|_{L^{1}(\bbR^{d})}+\left\| \frac{s^{n}w}{r^{n}}\right\|_{L^{1}(\bbR^{d})}\right]^{1/2}\left\| \frac{w}{r^{n}s^{m}}\right\|_{L^{\ift}(\bbR^{d})}^{1/2}.
		\end{equation}
		We start by splitting the integral range of $u^{r}$ and $u^{s}$ into
		\begin{equation}
			\begin{aligned}
				u^{r}(r,s)&=\left[\iint_{B_{1/4}(r,s)^{c}}+\iint_{B_{1/4}(r,s)}\right]F_{n,m}^{r}(r,s,\br{r},\br{s})w(\br{r},\br{s})d\br{r}d\br{s}=:(I_{1})+(I_{2}),\\
				u^{s}(r,s)&=\left[\iint_{B_{1/4}(r,s)^{c}}+\iint_{B_{1/4}(r,s)}\right]F_{n,m}^{s}(r,s,\br{r},\br{s})w(\br{r},\br{s})d\br{r}d\br{s}=:(I_{1}')+(I_{2}'),
			\end{aligned}
		\end{equation}
		where $B_{1/4}(r,s):=\{ (\br{r},\br{s})\in \Pi : D<1/4\}$ and $ B_{1/4}(r,s)^{c}:=\Pi\setminus B_{1/4}(r,s)=\{ (\br{r},\br{s})\in \Pi : D\geq 1/4\} $. 
		On $ B_{1/4}(r,s)^{c} $, we use \eqref{eq:Frweight} and \eqref{eq:Fsweight} from Lemma \ref{lem:FrFsweight}, and the estimates $r^{2}+s^{2}\leq1$, $D\gtrsim 1$, and
		$$ \br{r}^{2}+\br{s}^{2}\leq D^{2}+[r^{2}+s^{2}] \lesssim D^{2}+1\lesssim D^{2}. $$
		Then we have
		\begin{equation}
			\begin{aligned}
				|(I_{1})|&\lesssim \iint_{B_{1/4}(r,s)^{c}}\frac{r}{D^{2}}
				|w(\br{r},\br{s})|d\br{r}d\br{s}\lesssim \iint_{B_{1/4}(r,s)^{c}}\left(\frac{\br{r}^{n}\br{s}^{m}}{\br{r}^{n}\br{s}^{m}}\right)^{1/2}\frac{|w(\br{r},\br{s})|}{D^{2}}d\br{r}d\br{s}\\
				&\lesssim \iint_{B_{1/4}(r,s)^{c}}\frac{(\br{r}^{n}\br{s}^{m}|w(\br{r},\br{s})|)^{1/2}}{D^{2}}d\br{r}d\br{s}\left\| \frac{w}{r^{n}s^{m}}\right\|_{L^{\ift}(B_{1/4}(r,s)^{c})}^{1/2}\\
				&\lesssim \left(\iint_{B_{1/4}(r,s)^{c}}\frac{1}{D^{4}}d\br{r}d\br{s}\right)^{1/2}\left(\iint_{B_{1/4}(r,s)^{c}}\br{r}^{n}\br{s}^{m}|w(\br{r},\br{s})|d\br{r}d\br{s}\right)^{1/2}\left\| \frac{w}{r^{n}s^{m}}\right\|_{L^{\ift}(B_{1/4}(r,s)^{c})}^{1/2}.
			\end{aligned}
		\end{equation}
		Then we use the Young's inequality $\br{r}^{n}\br{s}^{m}\leq (n\br{r}^{n+m}+m\br{s}^{n+m})/(n+m)$ to get
		\begin{equation}
			\begin{aligned}
				|(I_{1})|&\lesssim \left(\iint_{B_{1/4}(r,s)^{c}}(\br{r}^{n+m}+\br{s}^{n+m})|w(\br{r},\br{s})|d\br{r}d\br{s}\right)^{1/2}\left\| \frac{w}{r^{n}s^{m}}\right\|_{L^{\ift}(B_{1/4}(r,s)^{c})}^{1/2}\\
				&\lesssim \left[\left\| r^{n+m}w\right\|_{L^{1}(B_{1/4}(r,s)^{c})}+\left\| s^{n+m}w\right\|_{L^{1}(B_{1/4}(r,s)^{c})}\right]^{1/2} \left\| \frac{w}{r^{n}s^{m}}\right\|_{L^{\ift}(B_{1/4}(r,s)^{c})}^{1/2}\\
				&\lesssim 
				\left[\left\| \frac{r^{m}w}{s^{m}}\right\|_{L^{1}(\bbR^{d})}+\left\| \frac{s^{n}w}{r^{n}}\right\|_{L^{1}(\bbR^{d})}\right]^{1/2}\left\| \frac{w}{r^{n}s^{m}}\right\|_{L^{\ift}(\bbR^{d})}^{1/2}.
			\end{aligned}
		\end{equation}
		Similarly, we have
		\begin{equation}
			\begin{aligned}
				|(I_{1}')|&\lesssim \iint_{B_{1/4}(r,s)^{c}}\frac{s}{D^{2}}|w(\br{r},\br{s})|d\br{r}d\br{s}\lesssim \iint_{B_{1/4}(r,s)^{c}}\left(\frac{\br{r}^{n}\br{s}^{m}}{\br{r}^{n}\br{s}^{m}}\right)^{1/2}\frac{|w(\br{r},\br{s})|}{D^{2}}d\br{r}d\br{s}\\
				&\lesssim \left[\left\| \frac{r^{m}w}{s^{m}}\right\|_{L^{1}(\bbR^{d})}+\left\| \frac{s^{n}w}{r^{n}}\right\|_{L^{1}(\bbR^{d})}\right]^{1/2}\left\| \frac{w}{r^{n}s^{m}}\right\|_{L^{\ift}(\bbR^{d})}^{1/2}.
			\end{aligned}
		\end{equation}
		On $ B_{1/4}(r,s) $, 
		we can use 
		\eqref{eq:Frslow} from Lemma \ref{lem:Frslow} and the estimates $ \br{r}^{n+m}+\br{s}^{n+m}\sim 1 $ and $\br{r}^{n}\br{s}^{m}\lesssim 1$ to get
		\begin{equation}
			\begin{aligned}
				|(I_{2})|, |(I_{2}')|&
				\lesssim \iint_{B_{1/4}(r,s)}\frac{|w(\br{r},\br{s})|}{D}d\br{r}d\br{s}\lesssim \nrm{w}_{L^{1}(B_{1/4}(r,s))}^{1/2}\nrm{w}_{L^{\ift}(B_{1/4}(r,s))}^{1/2}\\
				&\lesssim \nrm{(r^{n+m}+s^{n+m})w}_{L^{1}(B_{1/4}(r,s))}^{1/2}\left\| \frac{w}{r^{n}s^{m}}\right\|_{L^{\ift}(B_{1/4}(r,s))}^{1/2}\\
				&\lesssim \left[\left\| \frac{r^{m}w}{s^{m}}\right\|_{L^{1}(\bbR^{d})}+\left\| \frac{s^{n}w}{r^{n}}\right\|_{L^{1}(\bbR^{d})}\right]^{1/2}\left\| \frac{w}{r^{n}s^{m}}\right\|_{L^{\ift}(\bbR^{d})}^{1/2}.
			\end{aligned}
		\end{equation}
	\end{proof}

	\subsection{Radial moments}
	In this subsection, we consider the radial moments and their behavior in time. 
	Let $k\geq 1$ be an integer. Then we respectively define $r$-radial moment and $s$-radial moment as
	\begin{equation}
		P_{k}^{r}(t):=\iint_{\Pi}r^{k}|w(t,r,s)|drds,\quad P_{k}^{s}(t):=\iint_{\Pi}s^{k}|w(t,r,s)|drds.
	\end{equation}
	Also, we define
	\begin{equation}
		L_{k}(t):=P_{k}^{r}(t)+P_{k}^{s}(t).
	\end{equation}
	We recall the initial data assumptions as $u_0\in L^2(\mathbb{R}^d)$ and 
	\[ 
	w_{0}\in (L^{d,1}\cap L^{\ift})(\bbR^{d}), \qquad \frac{w_0}{r^ns^m}\in (L^1\cap L^{\infty}) (\mathbb{R}^d),
	\] 
	together with
	\[
	\frac{w_0}{r^n}, \, \frac{w_0}{s^m} \in L^{\infty}(\mathbb{R}^d), \qquad \frac{r^m w_0}{s^m}, \, \frac{s^n w_0}{r^n} \in L^1 (\mathbb{R}^d).
	\]
	These assumptions imply of those in Theorem \ref{thm:LWP}, and therefore, the local existence of a unique solution to \eqref{eq:vorteq} is guaranteed. 
	
Let $[0,T^{*}_d)$ be the maximal existence interval of the local solution $w$ from Theorem \ref{thm:LWP}. For every finite $T<T^{*}_d$, we have
\begin{align} \label{eq:bkm}
\int_0^{T} \left( \left\| \frac{u^r(t)}{r}  \right\|_{L^{\infty}(\mathbb{R}^d)} + \left\| \frac{u^s(t)}{s}  \right\|_{L^{\infty}(\mathbb{R}^d)} \right) dt < +\infty.
\end{align}
In particular, since
\begin{equation}
\left(\partial_t+u^r\partial_r+u^s\partial_s\right) \frac{w}{r^{n-k}s^{m}} = k\frac{u^r}{r}\frac{w}{r^{n-k}s^{m}},
\end{equation}
and \[ P_k^r(t)\simeq_d \left\| \frac{w}{r^{n-k}s^m} \right\|_{L^1(\mathbb R^d)}, \]
Gronwall's inequality yields
\begin{align} \label{eq:pr_fin} 
P_k^r(t) \lesssim_d P_k^r(0) \exp\left( k\int_0^t \left\| \frac{u^r(\tau)}{r} \right\|_{L^\infty(\mathbb R^d)} d\tau \right), \qquad 0\le t<T_d^*. 
\end{align}
Consequently, the finiteness of the moment $P_k^r(t)$ propagates throughout the interval $[0,T^*_d)$. A similar conclusion can be said for the finiteness of $P_k^s(t)$ on the interval $[0,T^*_d)$.
		
	\begin{lem} \label{lem:rsder}
	For $k\ge 1$, assuming the finiteness of the initial moment $P_k^r(0)<+\infty$, we have
		\begin{align} \label{eq:rder}
			\frac{d}{dt}P_{k}^{r}(t)&\leq k\iint_{\Pi}r^{k-1}|u^{r}(t,r,s)||w(t,r,s)|drds,  \qquad 0\le t<T^*_d.
		\end{align}
	Similarly, assuming $P_k^s(0)<+\infty$, we have 	
		\begin{align} \label{eq:sder}
		\frac{d}{dt}P_{k}^{s}(t)&\leq k\iint_{\Pi}s^{k-1}|u^{s}(t,r,s)||w(t,r,s)|drds, \qquad 0\le t<T^*_d.
		\end{align}
	\end{lem}
	
	\begin{proof}[Proof of Lemma \ref{lem:rsder}]
		We only provide the proof of \eqref{eq:rder}, and \eqref{eq:sder} follows similarly. 
		
		Following the approach of Lemma 3.1 in \cite{EYao25}, we first regularize $P^r_k$ by introducing 
		\[ P^{r,\varepsilon}_{k}(t) := \iint_{\Pi} r^k \left( \sqrt{w^2+\varepsilon^2}-\varepsilon \right) drds, \qquad \varepsilon>0, \] where $t\in [0,T^{*}_d)$.
		By the monotone convergence theorem, it is clear that $P^{r,\varepsilon}_k \nearrow P^r_k$ as $\varepsilon \to 0$. Differentiating $P^{r,\varepsilon}_k$ and applying the vorticity equation along with integration by parts, we obtain
		\begin{align} \label{eq:rder1}
			\begin{split}
				\frac{d}{dt} P^{r,\varepsilon}_k (t) &= \iint_{\Pi} r^{k} \frac{w}{\sqrt{w^2+\varepsilon^2}}\, \partial_tw \, drds = -\iint_{\Pi} r^{k} \frac{w}{\sqrt{w^2+			\varepsilon^2}} \, \nabla_{(r,s)} \cdot (\bfv w) \, drds \\
				&= k \iint_{\Pi} r^{k-1} \frac{u^rw^2}{\sqrt{w^2+\varepsilon^2}} \, drds + \iint_{\Pi} r^{k} w \bfv\cdot \nabla_{(r,s)} \left( \frac{w}{\sqrt{w^2+	\varepsilon^2}} \right) drds \\
				&=  k \iint_{\Pi} r^{k-1} \frac{u^rw^2}{\sqrt{w^2+\varepsilon^2}} \, drds + \iint_{\Pi} r^{k}\bfv \cdot \nabla_{(r,s)} \left(- \frac{\varepsilon^2}{\sqrt{w^2+\varepsilon^2}} \right) drds.		 
			\end{split}
		\end{align}
		Using the incompressibility condition \eqref{eq:divergencefree}, the last integral can be written as
		\begin{align*}
			\iint_{\Pi} r^{k}\bfv \cdot \nabla_{(r,s)} \left(\varepsilon -  \frac{\varepsilon^2}{\sqrt{w^2+\varepsilon^2}} \right) drds &= \iint_{\Pi} \left( \varepsilon - \frac{\varepsilon^2}{\sqrt{w^2+\varepsilon^2}} \right) \nabla_{(r,s)} \cdot (r^k\bfv) \, drds \\
			&= \iint_{\Pi} \left( \varepsilon - \frac{\varepsilon^2}{\sqrt{w^2+\varepsilon^2}}\right) \left( (k-n)r^{k-1}u^r - r^k\frac{m}{s}u^s \right) drds. 
		\end{align*}
		Substituting this into \eqref{eq:rder1} gives
		\begin{align*}
			\frac{d}{dt} P^{r,\varepsilon}_k (t) = k \iint_{\Pi} r^{k-1} \frac{u^rw^2}{\sqrt{w^2+\varepsilon^2}} \, drds + \iint_{\Pi} \left(\varepsilon - \frac{\varepsilon^2}{\sqrt{w^2+\varepsilon^2}}\right) \left( (k-n)r^{k-1}u^r - r^k\frac{m}{s}u^s \right) drds. 
		\end{align*}
		Denoting $\varphi_{\varepsilon}(z):=\varepsilon - \frac{\varepsilon^2}{\sqrt{z^2+\varepsilon^2}}$ and integrating over time $0\le t_1 < t_2<T^{*}$, we obtain
		\begin{align*}
			P^{r,\varepsilon}_k (t_2) - P^{r,\varepsilon}_k (t_1) \le \int_{t_1}^{t_2} \!k \! \iint_{\Pi} r^{k-1} |u^r||w|\, drds dt + \int_{t_1}^{t_2} \! \iint_{\Pi} r^k \varphi_{\varepsilon}(w) \left( \left|\frac{(k-n)u^r}{r}\right| + \left|\frac{mu^s}{s}\right| \right) drds dt.
		\end{align*}
				
		Noting that \[ 0\le \varphi_{\varepsilon}(z)\le \min\{\varepsilon,|z|\}, \] we first observe
		\begin{align*} 
		r^k\varphi_{\varepsilon}(w) \left( \left|\frac{(k-n)u^r}{r}\right| +\left|\frac{mu^s}{s}\right| \right) \leq r^k|w| \left(|k-n| \left\|\frac{u^r}{r}\right\|_{L^\infty(\mathbb R^d)} + m \left\|\frac{u^s}{s}\right\|_{L^\infty(\mathbb R^d)} \right). 
		\end{align*}
		When combined with \eqref{eq:bkm} and \eqref{eq:pr_fin}, the above expression becomes integrable  on $[t_1,t_2]\times\Pi$, uniformly in $\varepsilon$. Moreover, it converges to $0$ as $\varepsilon \to 0$. So, applying the dominated convergence theorem, we obtain  
		\[ 
		\lim_{\varepsilon\to0} \int_{t_1}^{t_2}\!\!\iint_{\Pi} r^k\varphi_{\varepsilon}(w) \left( \left|\frac{(k-n)u^r}{r}\right| + \left|\frac{mu^s}{s}\right| \right) drdsdt=0.
		\]
		
		Finally, using this and the monotone convergence theorem for the left-hand side, we pass to the limit $\varepsilon\to0$ to conclude that 
		\[ 
		P_k^r(t_2)-P_k^r(t_1) \le k\int_{t_1}^{t_2} \iint_{\Pi} r^{k-1}|u^r||w|\,drds\,dt, 
		\] 
		for all $0\le t_1<t_2<T^*_d$. This completes the proof of \eqref{eq:rder}.
	\end{proof}
	
Note that 
\begin{equation} 
P_{d-2}^{r}(t)\simeq_{d}\left\| \frac{r^{m}w}{s^{m}}\right\|_{L^{1}(\bbR^{d})},\quad P_{d-2}^{s}(t)\simeq_{d}\left\| \frac{s^{n}w}{r^{n}}\right\|_{L^{1}(\bbR^{d})}. 
\end{equation} 
Since these quantities are finite initially by our assumptions, Lemma \ref{lem:rsder} can be applied for $k=d-2$.
	\begin{lem}\label{lem:ddtPd-2}
	Under the assumptions of Theorem \ref{thm:GWP}, we have
		\begin{align}
			\frac{d}{dt}P_{d-2}^{r}(t)&\lesssim_{d,w_0} \nrm{u^{r}(t)}_{L^{\ift}(\bbR^{d})}^{(d-4)/(d-2)}P_{d-2}^{r}(t)^{(d-3)/(d-2)},\\
			\frac{d}{dt}P_{d-2}^{s}(t)&\lesssim_{d,w_0} \nrm{u^{s}(t)}_{L^{\ift}(\bbR^{d})}^{(d-4)/(d-2)}P_{d-2}^{s}(t)^{(d-3)/(d-2)},
		\end{align}
	for all $0\le t < T^{*}_d$.
	\end{lem}

	\begin{proof}[Proof of Lemma \ref{lem:ddtPd-2}] As noted above, we can apply Lemma \ref{lem:rsder} with $k=d-2$ on the maximal existence interval $[0,T_d^*)$. In particular, from \eqref{eq:rder}, we have
		\begin{equation}
			\begin{aligned}
				\frac{d}{dt}P_{d-2}^{r}(t)&\leq (d-2)\iint_{\Pi}r^{d-3}|u^{r}(t,r,s)||w(t,r,s)| \,drds\\
				&\le (d-2) \left\| \frac{w(t)}{r^{n}s^{m}}\right\|_{L^{\ift}}^{\frac{1}{d-2}} \left\| u^{r}(t) \right\|_{L^{\ift}}^{\frac{d-4}{d-2}} \iint_{\Pi}r^{d-3}r^{\frac{n}{d-2}}s^{\frac{m}{d-2}}|u^{r}(t,r,s)|^{\frac{2}{d-2}}|w(t,r,s)|^{\frac{d-3}{d-2}} \,drds\\
				&=(d-2) \left\| \frac{w_0}{r^{n}s^{m}}\right\|_{L^{\ift}}^{\frac{1}{d-2}}\nrm{u^{r}(t)}_{L^{\ift}}^{\frac{d-4}{d-2}} \iint_{\Pi} \left(r^{n}s^{m}|u^{r}(t,r,s)|^{2} \right)^{\frac{1}{d-2}} \left(r^{d-2}|w(t,r,s)|\right)^{\frac{d-3}{d-2}}drds.
			\end{aligned}
		\end{equation}
	Here, we recall $d\ge 4$ and the fact that $w/(r^ns^m)$ is transported by the flow. Then, applying H\"{o}lder's inequality gives	
		\begin{equation}
			\begin{aligned}
				\frac{d}{dt}P_{d-2}^{r}(t)&\le (d-2) \left\| \frac{w_0}{r^{n}s^{m}}\right\|_{L^{\ift}}^{\frac{1}{d-2}}\nrm{u^{r}(t)}_{L^{\ift}}^{\frac{d-4}{d-2}} \! \left(\iint_{\Pi}r^{n}s^{m}|u^{r}(t,r,s)|^{2}drds\right)^{\frac{1}{d-2}} \! \! \left( \iint_{\Pi} r^{d-2}|w(t,r,s)|drds\right)^{\frac{d-3}{d-2}}\\
				&\le (d-2) \left\| \frac{w_{0}}{r^{n}s^{m}}\right\|_{L^{\ift}}^{\frac{1}{d-2}}\left\| u^{r}(t) \right\|_{L^{\ift}}^{\frac{d-4}{d-2}}\left\| r^{n/2}s^{m/2}\bfv(t)\right\|_{L^{2}(\Pi)}^{\frac{2}{d-2}}P_{d-2}^{r}(t)^{\frac{d-3}{d-2}}\\
				&\simeq_d \left\| \frac{w_{0}}{r^{n}s^{m}}\right\|_{L^{\ift}(\bbR^{d})}^{\frac{1}{d-2}}\left\|u^{r}(t)\right\|_{L^{\ift}(\bbR^{d})}^{\frac{d-4}{d-2}}\left\| \bfv_{0} \right\|_{L^{2}(\bbR^{d})}^{\frac{2}{d-2}}P_{d-2}^{r}(t)^{\frac{d-3}{d-2}},
			\end{aligned}
		\end{equation}
	where, in the last line, we also used the conservation of the kinetic energy.	
		
	The estimate for $P_{d-2}^{s}(t)$ follows analogously. Indeed, using \eqref{eq:sder}, we first obtain
		\begin{equation}
			\begin{aligned}
				\frac{d}{dt}P_{d-2}^{s}(t)&\leq (d-2)\iint_{\Pi}s^{d-3}|u^{s}(t,r,s)||w(t,r,s)|drds\\
				&\le (d-2) \left\| \frac{w(t)}{r^{n}s^{m}}\right\|_{L^{\ift}}^{\frac{1}{d-2}} \left\| u^{s}(t) \right\|_{L^{\ift}}^{\frac{d-4}{d-2}} \iint_{\Pi}s^{d-3}r^{\frac{n}{d-2}}s^{\frac{m}{d-2}}|u^{s}(t,r,s)|^{\frac{2}{d-2}}|w(t,r,s)|^{\frac{d-3}{d-2}} \,drds\\
				&=(d-2) \left\| \frac{w_0}{r^{n}s^{m}}\right\|_{L^{\ift}}^{\frac{1}{d-2}}\nrm{u^{s}(t)}_{L^{\ift}}^{\frac{d-4}{d-2}} \iint_{\Pi} \left(r^{n}s^{m}|u^{s}(t,r,s)|^{2} \right)^{\frac{1}{d-2}} \left(s^{d-2}|w(t,r,s)|\right)^{\frac{d-3}{d-2}}drds. 
			\end{aligned}
		\end{equation}
	Consequently, applying H\"{o}lder's inequality along with the conservation of the kinetic energy, we conclude
		\begin{equation}
			\begin{aligned}
				\frac{d}{dt}P_{d-2}^{s}(t)&\le (d-2) \left\| \frac{w_0}{r^{n}s^{m}}\right\|_{L^{\ift}}^{\frac{1}{d-2}}\nrm{u^{s}(t)}_{L^{\ift}}^{\frac{d-4}{d-2}} \! \left(\iint_{\Pi}r^{n}s^{m}|u^{s}(t,r,s)|^{2}drds\right)^{\frac{1}{d-2}} \! \! \left( \iint_{\Pi} s^{d-2}|w(t,r,s)|drds\right)^{\frac{d-3}{d-2}}\\
				&= (d-2) \left\| \frac{w_{0}}{r^{n}s^{m}}\right\|_{L^{\ift}}^{\frac{1}{d-2}}\left\| u^{s}(t) \right\|_{L^{\ift}}^{\frac{d-4}{d-2}}\left\| r^{n/2}s^{m/2}\bfv(t)\right\|_{L^{2}(\Pi)}^{\frac{2}{d-2}}P_{d-2}^{s}(t)^{\frac{d-3}{d-2}}\\
				&\simeq_d \left\| \frac{w_{0}}{r^{n}s^{m}}\right\|_{L^{\ift}(\bbR^{d})}^{\frac{1}{d-2}}\left\|u^{s}(t)\right\|_{L^{\ift}(\bbR^{d})}^{\frac{d-4}{d-2}}\left\| \bfv_{0} \right\|_{L^{2}(\bbR^{d})}^{\frac{2}{d-2}}P_{d-2}^{s}(t)^{\frac{d-3}{d-2}}.
			\end{aligned}
		\end{equation}
	\end{proof}

The above lemma yields the following upper bounds for $L_{d-2}(t)=P_{d-2}^{r}(t)+P_{d-2}^{s}(t)$, and, in turn, for $\|\bfv\|_{L^{\infty}}$ on the maximal existence interval $[0,T^*_d)$.
	\begin{lem}\label{lem:Ld-2time} 
		Under the assumptions of Theorem \ref{thm:GWP}, for all $0\le t < T^{*}_d$, we have
		\begin{equation}\label{eq:Ld-2time}
			L_{d-2}(t)\lesssim_{d,w_{0}} \begin{cases}
				(1+t)^{2},&\quad\text{if}\quad d=4,\\
				(1+t)^{6},&\quad\text{if}\quad d=5,\\
				e^{ct},&\quad\text{if}\quad d=6,
			\end{cases}
		\end{equation}
		and moreover,
		\begin{equation}\label{eq:uttime}
			\nrm{\bfv(t)}_{L^{\ift}(\bbR^{d})}\lesssim_{d,w_{0}} \begin{cases}
				1+t,&\quad\text{if}\quad d=4,\\
				(1+t)^{3},&\quad\text{if}\quad d=5,\\
				e^{ct},&\quad\text{if}\quad d=6.
			\end{cases}
		\end{equation}
	\end{lem}

	\begin{proof}[Proof of Lemma \ref{lem:Ld-2time}]
		First, for every $0\le t<T_d^*$, summing up the estimates in Lemma \ref{lem:ddtPd-2}, we get 
		\begin{equation}
			\begin{aligned}
				\frac{d}{dt}L_{d-2}(t)&=\frac{d}{dt}\left[P_{d-2}^{r}(t)+P_{d-2}^{s}(t)\right] \lesssim_ {d,w_0} \nrm{\bfv(t)}_{L^{\ift}(\bbR^{d})}^{(d-4)/(d-2)}L_{d-2}(t)^{(d-3)/(d-2)}.
			\end{aligned}
		\end{equation}
		On the other hand, from the estimate \eqref{eq:FStype} in Proposition \ref{prop:FStypeu}, we have
		\begin{equation}\label{eq:utLd-2}
			\nrm{\bfv(t)}_{L^{\ift}(\bbR^{d})}\lesssim_d \left[P_{d-2}^{r}(t)+P_{d-2}^{s}(t)\right]^{1/2}\left\| \frac{w(t)}{r^{n}s^{m}}\right\|_{L^{\ift}(\bbR^{d})}^{1/2}=L_{d-2}(t)^{1/2}\left\| \frac{w_{0}}{r^{n}s^{m}}\right\|_{L^{\ift}(\bbR^{d})}^{1/2}.
		\end{equation}
		Plugging this into the previous estimate gives
		\begin{equation}
			\begin{aligned}
				\frac{d}{dt}L_{d-2}(t) \lesssim_{d,w_0} L_{d-2}(t)^{(3d-10)/(2d-4)}.
			\end{aligned}
		\end{equation}
		Solving this differential inequality for $4\le d \le 6$ yields \eqref{eq:Ld-2time}. 

		To prove the estimate \eqref{eq:uttime}, one simply needs to substitute \eqref{eq:Ld-2time} into the inequality \eqref{eq:utLd-2}.
	\end{proof}
	
We are now ready to prove Theorem \ref{thm:GWP}. Before proceeding, let us briefly explain the continuation argument. By Lemma \ref{lem:Ld-2time}, the estimate \eqref{eq:uttime} for $\|\bfv(t)\|_{L^\infty}$ is already available throughout the maximal existence interval $[0,T_d^*)$. Therefore, it remains to establish the estimate \eqref{eq:wttime} for $\|\omega(t)\|_{L^\infty}$ on $[0,T_d^*)$. Once this is achieved, the BKM criterion ensures that the solution can be continued beyond $T_d^*$, contradicting the maximality of $T_d^*$ unless $T_d^*=\infty$. Consequently, the local solution extends uniquely to a global-in-time solution; the estimates \eqref{eq:Ld-2time}, \eqref{eq:uttime}, and \eqref{eq:wttime} also hold globally in time.

\begin{proof}[Proof of Theorem \ref{thm:GWP}]
		Consider the flow map $\Phi_{t}=(\Phi_{t}^{r},\Phi_{t}^{s})$ as the unique solution of the ODE
		\begin{equation}
			\begin{aligned}
				\frac{d}{dt}\Phi_{t}^{r}(r,s)=u^{r}(t,\Phi_{t}^{r}(r,s),\Phi_{t}^{s}(r,s)),\quad \Phi_{0}^{r}(r,s)=r,\\
				\frac{d}{dt}\Phi_{t}^{s}(r,s)=u^{s}(t,\Phi_{t}^{r}(r,s),\Phi_{t}^{s}(r,s)),\quad \Phi_{0}^{s}(r,s)=s.
			\end{aligned}
		\end{equation}
		Then, we can write 
		\begin{equation}
			\Phi_{t}^{r}(r,s)=r+\int_{0}^{t}u^{r}(\tau,\Phi_{\tau}^{r}(r,s),\Phi_{\tau}^{s}(r,s))d\tau,\quad \Phi_{t}^{s}(r,s)=s+\int_{0}^{t}u^{s}(\tau,\Phi_{\tau}^{r}(r,s),\Phi_{\tau}^{s}(r,s))d\tau.
		\end{equation}
		Since $w/(r^ns^m)$ is conserved along particle trajectories, we have
		\begin{equation}
			\frac{|w(t,\Phi_{t}^{r}(r,s),\Phi_{t}^{s}(r,s))|}{[\Phi_{t}^{r}(r,s)]^{n}[\Phi_{t}^{s}(r,s)]^{m}}=\frac{|w_{0}(r,s)|}{r^{n}s^{m}},
		\end{equation}
		and therefore,
		\begin{equation}\label{eq:wtPhit}
			\begin{aligned}
				&|w(t,\Phi_{t}^{r}(r,s),\Phi_{t}^{s}(r,s))|=\frac{|w_{0}(r,s)|}{r^{n}s^{m}}[\Phi_{t}^{r}(r,s)]^{n}[\Phi_{t}^{s}(r,s)]^{m}\\
				\quad&=\frac{|w_{0}(r,s)|}{r^{n}s^{m}}\left(r+\int_{0}^{t}u^{r}(\tau,\Phi_{\tau}^{r}(r,s),\Phi_{\tau}^{s}(r,s))d\tau\right)^{n}\left(s+\int_{0}^{t}u^{s}(\tau,\Phi_{\tau}^{r}(r,s),\Phi_{\tau}^{s}(r,s))d\tau\right)^{m}\\
				&\leq\frac{|w_{0}(r,s)|}{r^{n}s^{m}}\left(r+\int_{0}^{t}\nrm{u^{r}(\tau)}_{L^{\ift}(\bbR^{d})}d\tau\right)^{n}\left(s+\int_{0}^{t}\nrm{u^{s}(\tau)}_{L^{\ift}(\bbR^{d})}d\tau\right)^{m}.
			\end{aligned}
		\end{equation}
		Then, using the estimate \eqref{eq:uttime} for $\nrm{u}_{L^{\ift}}$ in Lemma \ref{lem:Ld-2time},  for $0\le t<T^*_d$ we obtain
		\begin{equation}
			|w(t,\Phi_{t}^{r}(r,s),\Phi_{t}^{s}(r,s))|\leq\frac{|w_{0}(r,s)|}{r^{n}s^{m}}\left(r+C(d,w_{0})f(t)\right)^{n}\left(s+C(d,w_{0})f(t)\right)^{m},
		\end{equation}
		with
		\begin{equation}
			f(t):=\begin{cases}
				(1+t)^{2},&\quad\text{if}\quad d=4,\\
				(1+t)^{4},&\quad\text{if}\quad d=5,\\
				e^{ct},&\quad\text{if}\quad d=6.
			\end{cases}
		\end{equation}			

From the elementary inequality $(a+b)^k \le 2^{k-1}(a^k+b^k)$ that holds for all nonnegative $a,b$ and integer $k\ge 1$, we get
\begin{align*}
|w(t,\Phi_{t}^{r}(r,s),\Phi_{t}^{s}(r,s))|&\le 2^{d-4} \frac{|w_{0}(r,s)|}{r^{n}s^{m}} \left(r^n +C(d,\omega_0)^n f(t)^n \right) \left(s^m + C(d,\omega_0)^m f(t)^m\right) \\ 
&\lesssim_d |w_0(r,s)| + C(d,\omega_0)^m \frac{|w_{0}(r,s)|}{s^{m}} f(t)^m + C(d,\omega_0)^n \frac{|w_{0}(r,s)|}{r^{n}} f(t)^n \\
&\qquad \qquad \qquad+ C(d,\omega_0)^{m+n} \frac{|w_{0}(r,s)|}{r^{n}s^{m}} f(t)^{m+n}.
\end{align*}
By assumptions, we have $w_0, \, w_0/s^{m},\, w_0/r^{n},\, w_0/(r^{n}s^{m}) \in L^\infty(\mathbb R^d)$, and therefore, 
\begin{align*}
|w(t,\Phi_t^r(r,s),\Phi_t^s(r,s))| &\le C (d,w_0) \left(1+f(t)^m+f(t)^n+f(t)^{m+n} \right).
\end{align*}
Noting that $1\le f(t)$, we arrive at the desired estimate
\begin{align*}
\| w(t) \|_{L^{\infty}(\mathbb{R}^d)} \lesssim_ {d,\omega_0} f(t)^{m+n} = f(t)^{d-2}.
\end{align*}
for all $0\le t<T_d^*$. As it was noted earlier, by the BKM criterion, we conclude $T_d^{*}=\infty$ and the solution becomes global-in-time as well as the above estimate holds for all $0\le t<\infty$.

		This finishes the proof of Theorem \ref{thm:GWP}.
	\end{proof}

\ackn{D. Lim was supported by the National Research Foundation of Korea grant RS-2024-00350427. The authors would like to thank Professor In-Jee Jeong (KIAS) and Professor Yao Yao (NUS) for helpful discussions and comments.}


    
	

	\bibliographystyle{plain}
	\bibliography{biblography-L-260728}

\begin{thebibliography}{10}

\bibitem{AHK}
Hammadi Abidi, Taoufik Hmidi, and Sahbi Keraani.
\newblock On the global well-posedness for the axisymmetric {E}uler equations.
\newblock {\em Math. Ann.}, 347(1):15--41, 2010.

\bibitem{Chemin}
Jean-Yves Chemin.
\newblock Fluides parfaits incompressibles.
\newblock {\em Ast\'{e}risque}, (230):177, 1995.

\bibitem{Child07}
Stephen Childress.
\newblock Models of vorticity growth in {E}uler flows {I}, {A}xisymmetric flow
  without swirl and {I}{I}, {A}lmost 2-{D} dynamics.
\newblock {\em AML reports 05-07 and 06-07, Courant institute of mathematical
  sciences}, 2007.

\bibitem{Child08}
Stephen Childress.
\newblock Growth of anti-parallel vorticity in {E}uler flows.
\newblock {\em Phys. D}, 237(14-17):1921--1925, 2008.

\bibitem{ChilGil1}
Stephen Childress, Andrew~D. Gilbert, and Paul Valiant.
\newblock Eroding dipoles and vorticity growth for {E}uler flows in
  {$\Bbb{R}^3$}: axisymmetric flow without swirl.
\newblock {\em J. Fluid Mech.}, 805:1--30, 2016.

\bibitem{CJL_gwpbi}
Kyudong Choi, In-Jee Jeong, and Deokwoo Lim.
\newblock On global regularity of some bi-rotational {E}uler flows in
  {$\Bbb{R}^4$}.
\newblock {\em J. Differential Equations}, 425:627--660, 2025.

\bibitem{CJLglobal22}
Kyudong Choi, In-Jee Jeong, and Deokwoo Lim.
\newblock Global regularity for some axisymmetric {E}uler flows in {$\Bbb
  R^d$}.
\newblock {\em Proc. Amer. Math. Soc.}, 154(6):2605--2616, 2026.

\bibitem{Danaxi}
Rapha\"{e}l\ Danchin.
\newblock Axisymmetric incompressible flows with bounded vorticity.
\newblock {\em Uspekhi Mat. Nauk}, 62(3(375)):73--94, 2007.

\bibitem{DE}
Theodore~D. Drivas and Tarek~M. Elgindi.
\newblock Singularity formation in the incompressible {E}uler equation in
  finite and infinite time.
\newblock {\em EMS Surv. Math. Sci.}, 10(1):1--100, 2023.

\bibitem{EYao25}
Khakim Egamberganov and Yao Yao.
\newblock Growth estimates for axisymmetric euler equations without swirl.
\newblock {\em preprint, arXiv:2512.13456}.

\bibitem{FeSv}
Hao Feng and Vladim\'{i}r {\v{S}}ver{\'a}k.
\newblock On the {C}auchy problem for axi-symmetric vortex rings.
\newblock {\em Arch. Ration. Mech. Anal.}, 215:89--123, 2015.

\bibitem{GMT2023}
Stephen Gustafson, Evan Miller, and Tai-Peng Tsai.
\newblock Growth rates for anti-parallel vortex tube {E}uler flows in three and
  higher dimensions.
\newblock {\em J. Math. Fluid Mech.}, 28(1):Paper No. 19, 16, 2026.

\bibitem{HouZhang22P1}
Thomas~Y. Hou and Shumao Zhang.
\newblock Potential singularity of the axisymmetric euler equations with
  $c^\alpha$ initial vorticity for a large range of $\alpha$. part i: the
  $3$-dimensional case.
\newblock 2022.

\bibitem{HouZhang22P2}
Thomas~Y. Hou and Shumao Zhang.
\newblock Potential singularity of the axisymmetric euler equations with
  $c^\alpha$ initial vorticity for a large range of $\alpha$. part ii: the
  $n$-dimensional case.
\newblock 2022.

\bibitem{JL25}
In-Jee Jeong and Deokwoo Lim.
\newblock Support growth of vorticity for bi-rotational euler flows in high
  dimensions.
\newblock {\em Rev. Mat. Iberoam., to appear, arXiv:2510.18335}.

\bibitem{KhYa}
Boris Khesin and Cheng Yang.
\newblock Higher-dimensional {E}uler fluids and {H}asimoto transform:
  counterexamples and generalizations.
\newblock {\em Nonlinearity}, 34(3):1525--1542, 2021.

\bibitem{Limglobal23}
Deokwoo Lim.
\newblock Global regularity of some axisymmetric, single-signed vorticity in
  any dimension.
\newblock {\em Commun. Math. Sci.}, 23(1):279--298, 2025.

\bibitem{LJ_optimal}
Deokwoo Lim and In-Jee Jeong.
\newblock On the optimal rate of vortex stretching for axisymmetric {E}uler
  flows without swirl.
\newblock {\em Arch. Ration. Mech. Anal.}, 249(3):Paper No. 32, 31, 2025.

\bibitem{Maffei2001}
Carlotta Maffei and Carlo Marchioro.
\newblock A confinement result for axisymmetric fluids.
\newblock {\em Rend. Sem. Mat. Univ. Padova}, 105:125--137, 2001.

\bibitem{Majda2002}
Andrew~J. Majda and Andrea~L. Bertozzi.
\newblock {\em Vorticity and incompressible flow}, volume~27 of {\em Cambridge
  Texts in Applied Mathematics}.
\newblock Cambridge University Press, Cambridge, 2002.

\bibitem{Miller26prep}
Evan Miller.
\newblock Global regularity for axisymmetric, swirl-free solutions of the euler
  equation in four dimensions.
\newblock {\em preprint, arXiv:2602.13963}.

\bibitem{Miller}
Evan Miller and Tai-Peng Tsai.
\newblock On the regularity of axisymmetric, swirl-free solutions of the euler
  equation in four and higher dimensions.
\newblock {\em Rev. Mat. Iberoam., to appear, arXiv:2204.13406}.

\bibitem{Raymond}
X.~Saint~Raymond.
\newblock Remarks on axisymmetric solutions of the incompressible {E}uler
  system.
\newblock {\em Comm. Partial Differential Equations}, 19(1-2):321--334, 1994.

\bibitem{SWZ25}
Feng Shao, Dongyi Wei, and Zhifei Zhang.
\newblock Global regularity of axisymmetric {E}uler equations without swirl in
  higher dimensions.
\newblock {\em Acta Math. Sin. (Engl. Ser.)}, 42(3):663--679, 2026.

\bibitem{SY}
Taira Shirota and Taku Yanagisawa.
\newblock Note on global existence for axially symmetric solutions of the
  {E}uler system.
\newblock {\em Proc. Japan Acad. Ser. A Math. Sci.}, 70(10):299--304, 1994.

\bibitem{UY1968}
M.~R. Ukhovskii and V.~I. Yudovich.
\newblock Axially symmetric flows of ideal and viscous fluids filling the whole
  space.
\newblock {\em J. Appl. Math. Mech.}, 32:52--61, 1968.

\bibitem{Yang}
Cheng Yang.
\newblock Vortex motion of the {E}uler and lake equations.
\newblock {\em J. Nonlinear Sci.}, 31(3):Paper No. 48, 21, 2021.

\end{thebibliography}
	
\end{document}